\documentclass[UTF8,oneside,10pt]{article}

\usepackage[a4paper,margin=1.3in]{geometry}

\usepackage{amsmath}           
\usepackage{amssymb}           
\usepackage{mathrsfs}          
\usepackage{eucal}             
\usepackage{amsthm}            
\usepackage{mathtools}         
\usepackage{stmaryrd}          
\usepackage{arcs}              
\usepackage{esint}             
\usepackage{extarrows}         
\usepackage{enumerate}         
\usepackage{xcolor}            
\usepackage{graphicx}          
\usepackage{tikz}              
\usepackage{tikz-3dplot}       
\usepackage{tikz-cd}           
\usepackage[all,cmtip]{xy}     
\usepackage{pgfplots}          
\pgfplotsset{compat=newest}    
\usepackage{subcaption}        
\usepackage[ocgcolorlinks,linkcolor=blue]{hyperref}
\usepackage[capitalise]{cleveref}          
\usepackage[hyperref=true,backend=biber,style=alphabetic,backref=true,url=false]{biblatex}

\usepackage{tcolorbox}         
\tcbuselibrary{most}           
\usepackage{bm}                
\usepackage{slashed}           

\usepackage[utf8]{inputenc}
\usepackage[T1]{fontenc}

\usepackage{svg}               

\usepackage{multicol}          

\usepackage{amsmath,amsthm,amsfonts,amssymb,bm}

\usepackage{fourier}

\usepackage{bbold}

\renewcommand{\emptyset}{\varnothing}

\usepackage{graphicx} 
\usepackage{float} 

\theoremstyle{plain}\newtheorem{theorem}{Theorem}[section]
\theoremstyle{definition}\newtheorem{definition}[theorem]{Definition}
\theoremstyle{plain}
\theoremstyle{plain}\newtheorem{corollary}[theorem]{Corollary}
\theoremstyle{plain}\newtheorem{lemma}[theorem]{Lemma}
\theoremstyle{plain}\newtheorem{sublemma}[theorem]{Sublemma}
\theoremstyle{plain}\newtheorem{proposition}[theorem]{Proposition}
\theoremstyle{plain}
\theoremstyle{plain}
\theoremstyle{plain}
\theoremstyle{remark}
\theoremstyle{remark}
\theoremstyle{remark}
\theoremstyle{plain}
\theoremstyle{plain}
\theoremstyle{plain}
\theoremstyle{remark}\newtheorem{remark}[theorem]{Remark}
\theoremstyle{remark}
\theoremstyle{remark}
\numberwithin{equation}{section}
\numberwithin{theorem}{section}

\usepackage[backend=biber,style = alphabetic]{biblatex}
\usepackage{xcolor}

\definecolor{c1}{HTML}{88d5d2} 
\definecolor{c2}{HTML}{9c9d47} 
\definecolor{c3}{HTML}{fec842} 
\definecolor{c4}{HTML}{e97a2e} 
\definecolor{c5}{HTML}{834e71} 

\begin{document}

\begin{center}
{\Large \textbf{A Linear Bound on the Diameter of the Kakimizu Complex for Hyperbolic Knots}}

{\Large \textbf{ }} 

{\large \textsc{Xiao} CHEN$^{*}$ and \textsc{Wujie} SHEN$^{**}$} 

{\small Tsinghua University}

{\footnotesize \emph{e-mail:} x-chen20@tsinghua.org.cn$^{*}$}

{\footnotesize \emph{e-mail:} shenwj22@mails.tsinghua.edu.cn$^{**}$}

\end{center}

\begin{abstract}

    This paper focuses on the Kakimizu complex of a hyperbolic knot $K$. We define a complex $IS_\ell(K)$ to study incompressible Seifert surfaces of genus at most $\ell$, and prove that it is connected and that its diameter admits a linear upper bound in terms of $\ell$. As a corollary, we show that the diameter of the Kakimizu complex $MS(K)$ of a hyperbolic knot grows linearly with the genus $g$, confirming a conjecture of Sakuma--Shackleton. More precisely, it is bounded above by $6g-4$.

\end{abstract}




\section{Introduction} \label{Diameter Kakimizu section: Introduction}

Given a knot $K$ (or a link $L$) in $\mathbb{S}^3$, a \emph{spanning surface} is an embedded (possibly disconnected) compact surface whose boundary is exactly $K$ (or $L$). A connected, orientable spanning surface is called a \emph{Seifert surface}. A Seifert surface of minimal genus is called a \emph{minimal-genus Seifert surface}, and its genus is referred to as the \emph{genus of the knot (or link)}, denoted by $g(K)$ (often abbreviated to $g$ in this paper). The Kakimizu complex $MS(K)$ of a knot $K \subset \mathbb{S}^3$, introduced by Kakimizu in \cite{MR1177053}, encodes both the classification of minimal-genus Seifert surfaces for $K$ and the topological adjacency relations among them.

The topology and simplicial structure of $MS(K)$ have been extensively studied: Kakimizu showed that it is connected \cite{MR1177053}, Schultens proved that it is simply connected \cite{MR2746341}, and Przytycki--Schultens later established its contractibility \cite{MR2869183}. Moreover, Agol--Zhang showed that every saturated simplex of $MS(K)$ has maximal dimension \cite{agol2022guts}. In certain families of knots, the complex can even be computed explicitly \cite{MR1177053, MR2131376, MR1315011}.

In the study of the Kakimizu complex diameter, Pelayo \cite{MR3078414} and Sakuma--Shackleton \cite{MR2531146}, using different methods, established a quadratic upper bound for the diameter of the Kakimizu complex of an atoroidal knot in terms of the genus of the knot. In this paper, we improve this result by proving a linear upper bound for the diameter of the Kakimizu complex of an atoroidal knot, thereby resolving a conjecture posed by Sakuma--Shackleton \cite[Question 1.7]{MR2531146}:

\begin{theorem}\label{main thm: Hyperbolic Knot Linear Bound}
Suppose $K$ is an atoroidal knot in $\mathbb{S}^3$ of genus $g$. Then
\begin{equation}
\operatorname{diam} (MS(K)) \leq 6g-4.
\end{equation}
\end{theorem}

The example from \cite[Proposition~1.4]{MR2531146} shows that the bound is sharp when $g=1$, where the diameter equals $2$. Moreover, the bound is optimal in order, as there exist atoroidal knots whose Kakimizu complexes have diameter at least $2g-1$ \cite[Proposition~1.6]{MR2531146}.

According to Thurston's trichotomy theorem \cite[Proposition 11.2.11]{martelli2016introduction}, atoroidal knots include both torus knots and hyperbolic knots. The remaining knots, those that are not atoroidal, are referred to as satellite knots. A torus knot has a unique incompressible Seifert surface \cite{MR1259522}, so its Kakimizu complex is trivial. Therefore, \cref{main thm: Hyperbolic Knot Linear Bound} primarily focuses on hyperbolic knots. Moreover, for hyperbolic knots, every vertex in the Kakimizu complex has finite valence \cite[Theorem 1.8]{MR2821430}. Combining these facts, hyperbolic knots admit only finitely many minimal-genus Seifert surfaces (however, the number of minimal-genus Seifert surfaces of a hyperbolic knot cannot be bounded solely as a function of the genus, as shown by Tsutsumi \cite[Theorem~1.1]{MR2426806}). 

In contrast, the properties mentioned above do not necessarily hold for satellite knots: the isotopy classes of their Seifert surfaces may form an infinite set \cite{MR440528}, the Kakimizu complex may have infinite diameter \cite{MR1177053}, and some vertices may have infinite valence \cite{MR2821430}.

Beyond the complex $MS(K)$ consisting of minimal-genus Seifert surfaces, Kakimizu \cite{MR1177053} also introduced a complex $IS(K)$ generated by incompressible Seifert surfaces. Although every minimal-genus Seifert surface is incompressible, the complexes $MS(K)$ and $IS(K)$ do not always coincide. For instance, certain \emph{pretzel knots} are hyperbolic and admit infinitely many incompressible Seifert surfaces of unbounded genus \cite{MR2420023}, whereas a hyperbolic knot admits only finitely many minimal-genus Seifert surfaces.

To build a bridge between $MS(K)$ and $IS(K)$, for any integer $\ell \geq g$, we define an intermediate complex $MS(K) \subset IS_{\ell}(K) \subset IS(K)$ called the \emph{$\ell$-Kakimizu complex}, consisting of incompressible Seifert surfaces of genus at most $\ell$. By definition, we have $MS(K)=IS_g(K)$ and $IS(K) = \cup_{\ell=g}^{\infty} IS_{\ell}(K)$. For the $\ell$-Kakimizu complex, we show that its diameter also admits a linear upper bound in terms of $\ell$:

\begin{theorem}\label{main thm2: Hyperbolic Knot Linear Bound}
Suppose that $ K \subset \mathbb{S}^3 $ is an atoroidal knot of genus $ g $. Then for any integer $ \ell \geq g $, the complex $ IS_{\ell}(K) $ is connected, and its diameter admits a linear upper bound in terms of $ \ell $:
\begin{align}
    \operatorname{diam}(IS_{\ell}(K)) \leq 2\max\{4\ell + g - 2, 3\ell + 3g - 4\}.
\end{align}
\end{theorem}


Our approach is inspired by the work of Sakuma--Shackleton \cite{MR2531146}, who analysed the structure of lifts of Seifert surfaces in the infinite cyclic cover of $\mathbb{S}^3 \setminus K$. 
Concretely, suppose that an atoroidal knot $K$ admits two distinct minimal-genus Seifert surfaces $S_1$ and $S_2$. Fix a lift of $S_1$ to the infinite cyclic cover. Then the lifts of $S_2$ decompose this lift of $S_1$ into a linearly ordered sequence of subsurfaces, and Sakuma--Shackleton established the following properties:
\begin{itemize}
    \item At most $2g-1$ of these subsurfaces have negative Euler characteristic;
    \item no sequence of $3g-2$ consecutive subsurfaces can all have vanishing Euler characteristic, since such a configuration would produce an incompressible torus, contradicting atoroidality.
\end{itemize}
While the initial observations yield only a quadratic upper bound for $\operatorname{diam}(MS(K))$, the key new ingredient of this paper is a refinement of this analysis via \emph{annulus surgery}. Using this technique, we show that all subsurfaces with negative Euler characteristic are contained in a sequence of consecutive subsurfaces of linear size in the genus, which we call the \emph{central zone} (see \cref{Diameter Kakimizu section: Central Zone}). This additional structure allows us to control the configuration much more effectively and reduces the resulting diameter bound to linear order.

\section{Background} \label{Diameter Kakimizu section: Background}

\subsection{Kakimizu's complexes}

\begin{definition}[Kakimizu complex, \cite{MR1177053}]
    Given a knot $K\subset \mathbb{S}^3$, the \emph{Kakimizu complex (or minimal-genus Seifert surface complex)} $MS(K)$ is a simplicial complex whose $0$-simplices are isotopy classes of orientable minimal-genus Seifert surfaces, and whose $n$-simplices ($n \geq 1$) are collections of $n+1$ distinct disjoint such classes (i.e., their representatives can be made disjoint through isotopy). These simplices are glued together along common faces in the usual way, forming the simplicial complex $MS(K)$.
\end{definition}

\begin{definition}[Incompressible Seifert surface complex, \cite{MR1177053}]
    The \emph{incompressible Seifert surface complex (or $\infty$-Kakimizu complex)} $IS(K)$ (or denoted by $IS_{\infty}(K)$) is a simplicial complex whose $0$-simplices are isotopy classes of orientable incompressible Seifert surfaces, and whose $n$-simplices ($n \geq 1$) are collections of $n+1$ distinct disjoint such classes.
\end{definition}

\begin{definition}[$\ell$-Kakimizu complex]
    We generalise the definition of the Kakimizu complex. For a knot $K \subset \mathbb{S}^3$ and an integer $\ell \geq g(K)$, where $g$ is the genus of $K$, the \emph{$\ell$-Kakimizu complex} $IS_\ell(K)$ is defined as the simplicial complex whose $0$-simplices are isotopy classes of orientable Seifert surfaces of genus at most $\ell$, and whose $n$-simplices ($n \geq 1$) are collections of $n+1$ distinct such classes that can be represented by disjoint surfaces.
\end{definition}

\begin{remark}
    For a knot $K$, every minimal-genus Seifert surface is incompressible. Therefore, the Kakimizu complex $MS(K)$ coincides with $IS_{g}(K)$ and is a subcomplex of the incompressible Seifert surface complex $IS_{\infty}(K) = \cup_{\ell=g}^{\infty} IS_{\ell}(K)$.
\end{remark}

\begin{definition}[Natural distance, \cite{MR2531146}]
    Let $K$ be a knot with two incompressible Seifert surfaces $S$ and $S'$, and suppose that their genera are both at most $\ell$. Their isotopy classes are denoted by $[S]$ and $[S']$, so that $[S]$ and $[S']$ are two vertices in $IS_{\infty}(K)$ (or $IS_{\ell}(K)$). The \emph{natural distance} $d_{\infty}([S],[S'])$ (or $d_{\ell}([S],[S'])$) is defined to be the distance between $[S]$ and $[S']$ in $IS_{\infty}(K)$ (or $IS_{\ell}(K)$).
    
    More precisely, let $d([S],[S'])$ denote either $d_\infty([S],[S'])$ or $d_\ell([S],[S'])$. If $d([S],[S']) = d$, then there exists a sequence of $d+1$ distinct incompressible Seifert surfaces $S = S_0, S_1, \cdots, S_d = S'$, each of genus at most $\ell$ in the case of $d_\ell$ (with no genus restriction in the case of $d_\infty$), such that for every $1 \leq i \leq d$, the surfaces $S_{i-1}$ and $S_i$ are disjoint. Moreover, $d$ is the minimal integer for which such a sequence exists.

\end{definition}

\begin{definition}[Diameter]
    For a knot $K \subset \mathbb{S}^3$, the \emph{diameter} of $IS_\ell(K)$ (or $IS_{\infty}(K)$) is defined to be the maximal natural distance between any two vertices in the complex if this quantity is finite; otherwise, the diameter is defined to be infinite.
\end{definition}

\subsection{Atoroidal knots and hyperbolic knots}

\begin{definition}[Incompressible torus]
    Let $M$ be a three-dimensional manifold. An embedded torus $f: \mathbb{T}^2 \hookrightarrow M$ is said to be \emph{incompressible} if the induced homomorphism $f_{*}: \pi_1 (\mathbb{T}^2) \rightarrow \pi_1(M)$ is injective.
\end{definition}

\begin{definition}[Essential torus]
    Let $M$ be a three-dimensional manifold. An embedded torus $\mathbb{T}^2$ is said to be \emph{essential} if it is an incompressible torus that is not boundary-parallel in $M$.
\end{definition}

\begin{definition}[Atoroidal knots]
    Let $K$ be a knot in $\mathbb{S}^3$, and let $W(K)$ be an open tubular neighbourhood of $K$. The knot $K$ is said to be \emph{atoroidal} if the $3$-manifold $\mathbb{S}^3 \setminus W(K)$ with torus boundary contains no essential torus. A knot is called a \emph{satellite knot} if it is not atoroidal.
\end{definition}

\begin{remark} \label{Thurston trichotomy theorem}
    Thurston's trichotomy theorem states that every atoroidal knot is either a torus knot or a hyperbolic knot \cite[Proposition 11.2.11]{martelli2016introduction}. Therefore, every knot is either a torus knot, a hyperbolic knot, or a satellite knot.
\end{remark}

\begin{remark} \label{rmk: torus kakimizu is trivial}
For a torus knot $K$, since it is fibered, it admits a unique minimal-genus Seifert surface given by the fiber. Consequently, its Kakimizu complex $MS(K)$ is trivial. Moreover, by \cite{MR1259522}, a torus knot admits a unique incompressible Seifert surface up to isotopy, and hence both $IS_\ell(K)$ and $IS(K)$ are trivial. Therefore, when considering upper bounds on the diameter of $MS(K)$, $IS_\ell(K)$, and $IS(K)$ for atoroidal knots, it suffices to restrict attention to the hyperbolic case.
\end{remark}

\section{Intersection of Two Seifert Surfaces} \label{Diameter Kakimizu section: Intersection of Two Seifert Surfaces}

To understand the intersection of two transversely intersecting Seifert surfaces, Sakuma--Shackleton introduced two lemmas, \cref{Sakuma1} and \cref{Sakuma2}. Moreover, we relax the assumptions of their second lemma (\cref{Sakuma2}) by proving \cref{Sakuma3}, from which \cref{Sakuma2} follows as a direct corollary.

\begin{lemma}[\cite{MR2531146}]\label{Sakuma1}
    Suppose that $S_1$ and $S_2$ are two distinct incompressible Seifert surfaces for a knot $K$. Then, $S_2$ is isotopic to a third surface $S_3$ such that $S_1 \cap S_3$ is a disjoint union of simple loops, and the number of connected components of $S_1 \cap S_3$ is less than or equal to that of $S_1 \cap S_2$.
\end{lemma}

\begin{definition}[Favourable pair]
    \label{defi: Favourable pair}
Let $S_1$ and $S_2$ be two distinct Seifert surfaces. We say that the pair $(S_1,S_2)$ is a \emph{favourable pair} if
\begin{itemize}
    \item $S_1$ and $S_2$ are incompressible;
    \item $S_1$ has minimal genus;
    \item $S_1$ and $S_2$ intersect only in a disjoint union of simple loops;
    \item every intersection loop is not homotopically trivial on $S_1$;
    \item except for $K$, every intersection loop is \emph{essential} (i.e. not homotopically trivial and not boundary-parallel) on $S_2$.
\end{itemize}
\end{definition}

\begin{sublemma}[{\cite[Theorem 3.1]{MR1630563}}]\label{FenleySublemma}
    Suppose that $S$ is a minimal-genus Seifert surface for a \emph{hyperbolic} knot $K$. If a simple loop $\gamma \subset S$ is isotopic to $K$ in $\mathbb{S}^3 \setminus K$, then $\gamma$ is also isotopic to $K$ on $S$.
\end{sublemma}

\begin{lemma}\label{Sakuma3}
    Suppose that $S_1$ and $S_2$ are two distinct incompressible Seifert surfaces for an atoroidal knot $K$, and that $S_1$ has minimal genus. Then $S_2$ is isotopic to a third surface $S_3$ such that $(S_1,S_3)$ is a favourable pair, and the number of connected components of $S_1 \cap S_3$ is less than or equal to that of $S_1 \cap S_2$.
\end{lemma}

\begin{proof}
    This argument is essentially adapted from the proof of \cref{Sakuma2} by Sakuma--Shackleton. However, this lemma is under weaker assumptions, and \cref{Sakuma2} follows as a direct corollary.

    Since a torus knot admits a unique incompressible Seifert surface up to isotopy (see \cref{rmk: torus kakimizu is trivial}), we henceforth restrict attention to hyperbolic knots. By \cref{Sakuma1}, we may assume that $S_1 \cap S_2$ is a disjoint union of simple loops. Suppose that there exists a loop $\gamma \subset (S_1 \cap S_2) \setminus K$ that is inessential on $S_2$; that is, $\gamma$ is either null-homotopic or boundary-parallel on $S_2$.

    If $\gamma$ is null-homotopic on $S_2$, then by the incompressibility of $S_1$, it must also be null-homotopic on $S_1$. Since $\mathbb{S}^3 \setminus K$ is irreducible, we can perform an isotopy that pushes a small neighbourhood of the disk bounded by $\gamma$ on $S_2$ across the disk bounded by $\gamma$ on $S_1$, thereby eliminating the intersection loop $\gamma$.

    If $\gamma$ is boundary-parallel on $S_2$, without loss of generality, we may assume that $\gamma$ is the boundary-parallel loop on $S_2$ that is closest to the knot $K$. Since $S_1$ has minimal genus, by \cref{FenleySublemma}, the loop $\gamma$ must also be boundary-parallel on $S_1$. We may suppose that $\gamma$ and $K$ cobound two annuli $A_1 \subset S_1$ and $A_2 \subset S_2$. The union $A_1 \cup A_2$, denoted by $\mathbb{T}$, is a torus in $\mathbb{S}^3$ such that $K \subset \mathbb{T}$. Moreover, since $\gamma$ is chosen to be the boundary-parallel loop on $S_2$ closest to the knot $K$, the annulus $A_2$ intersects $A_1$ only along $K$ and $\gamma$. Hence, $\mathbb{T}$ is an embedded torus. Cutting $\mathbb{S}^3$ along $\mathbb{T}$ results in two components, $X_1$ and $X_2$. In what follows, we assume (without loss of generality) that $X_1$ is a solid torus (see, e.g., \cite[Section~1.1, Exercise~2]{hatcher2007notes}).
    
    Since $K$ is a hyperbolic knot, it is isotopic to the core of $X_1$. Otherwise, either $\mathbb{T}$ would be a standardly embedded torus, implying that $K$ is a torus knot; or $\mathbb{T}$ would be an essential torus, implying that $K$ is a satellite knot. Both cases contradict that the knot is hyperbolic. Hence, the pair $(X_1, A_1)$ can be given a product structure $\left(A_1 \times[0,1], A_1 \times\{0\}\right)$ (see \cref{fig: product structure}). We can perform an isotopy that pushes a small neighbourhood of $A_2$ on $S_2$ across $A_1$ on $S_1$, thereby eliminating the intersection loop $\gamma$.

    \begin{figure}[H]
    \centering
    \includegraphics[width=1.0\textwidth]{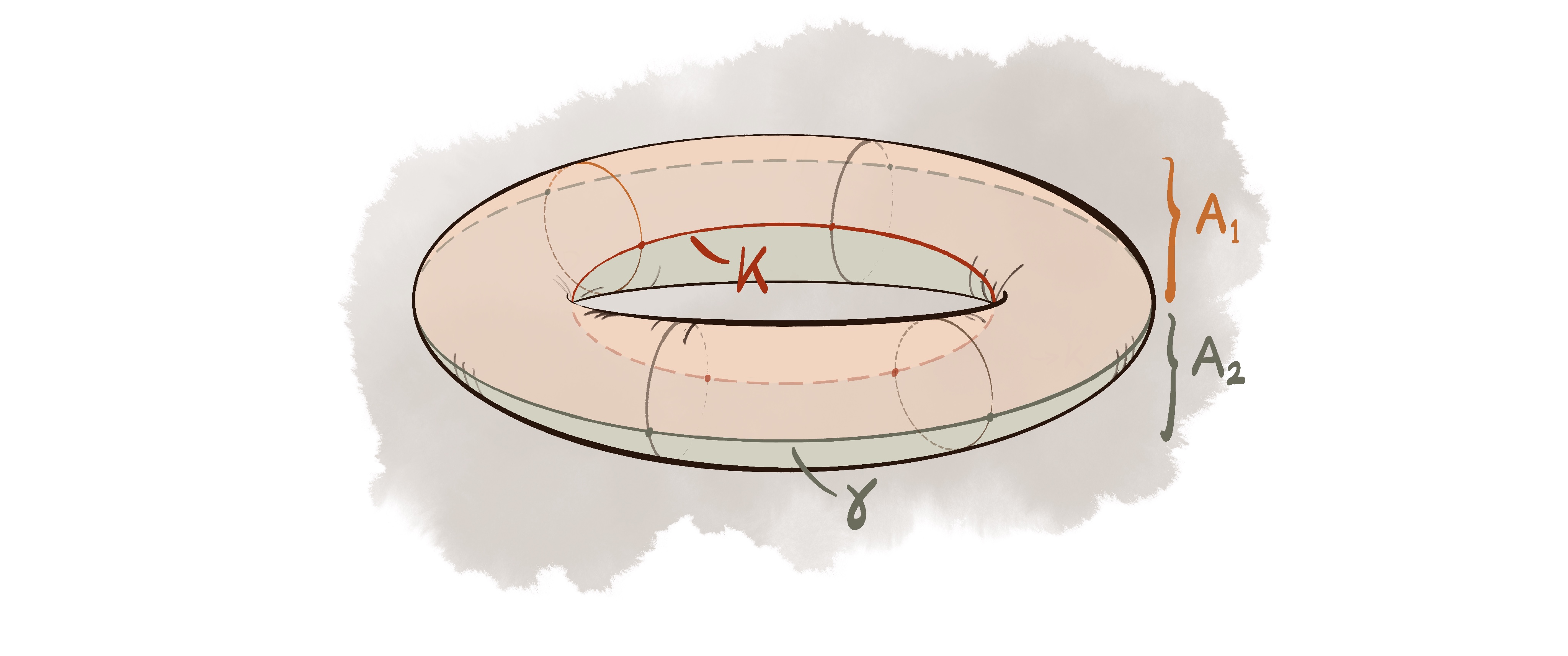}
    \caption{The pair $(X_1, A_1)$ can be given a product structure $\left(A_1 \times[0,1], A_1 \times\{0\}\right)$.}
    \label{fig: product structure}
    \end{figure}

    Symmetrically, if there is an intersection loop that is null-homotopic on $S_1$, then by an isotopy of $S_2$ we can eliminate this loop.
    
    After performing all the above isotopies of $S_2$, denote the resulting surface by $S_3$. Then every component of $S_1 \cap S_3 \setminus K$ is essential on $S_3$ and is not homotopically trivial on $S_1$. Therefore, the pair $(S_1,S_3)$ is favourable.
\end{proof}

As an immediate consequence of this lemma, we obtain the following corollary. In fact, this result already appears as \cite[Lemma~3.2]{MR2531146}.

\begin{corollary}[{\cite[Lemma 3.2]{MR2531146}}]\label{Sakuma2}
    Suppose that $S_1$ and $S_2$ are two distinct minimal-genus Seifert surfaces for an atoroidal knot $K$. Then $S_2$ is isotopic to a third surface $S_3$ such that both $(S_1,S_3)$ and $(S_3,S_1)$ are favourable pairs (that is, every intersection loop is essential on both $S_1$ and $S_3$), and the number of connected components of $S_1 \cap S_3$ is less than or equal to that of $S_1 \cap S_2$.
\end{corollary}

\begin{remark}
    In the definition of a favourable pair, we require that except for $K$, every intersection loop is essential on $S_2$, in particular, not boundary-parallel on $S_2$. This requirement is necessary because this property is used essentially in \cref{lem: not boundary-parallel} and \cref{pro: middle length}. Moreover, since it suffices for only one of the two surfaces to satisfy this condition, the definition of a favourable pair is chosen to be just strong enough for our purposes.

    To ensure that this property holds up to isotopy, we must require that $S_1$ has minimal genus; otherwise, we cannot always arrange for $S_2$ to satisfy the above condition. Indeed, \cite[Theorem 1.1]{MR2031451} provides infinitely many examples of incompressible but non-minimal-genus Seifert surfaces for a hyperbolic knot which contain a loop $\gamma$ that is not boundary-parallel on $S_1$ but is boundary-parallel in $\mathbb{S}^3 \setminus K$. Such surfaces are referred to as \emph{accidental}. Consequently, if $S_1$ does not have minimal genus and is such an accidental surface, and if $S_2$ intersects $S_1$ along the loop $\gamma$, then, provided that $S_2$ is not accidental, this loop may become parallel to $K$ on $S_2$.
\end{remark}

\section{A Technical Lemma}

In this section, we introduce a technical lemma that will be used in the proof of \cref{thm:euler-inequality}.

\begin{lemma}\label{lem: not boundary-parallel}
    Let $K$ be an atoroidal knot, and let $(S_1,S_2)$ be a favourable pair. Suppose that there exist two connected components $C_1$ and $C_2$ of $S_1 \cap S_2$, both of which are simple loops, such that $C_1$ and $C_2$ are isotopic on $S_1$ as well as on $S_2$. For $i=1,2$, let $A_i \subset S_i$ be the annulus bounded by $C_1$ and $C_2$. We define $\mathbb{T} := A_1 \cup A_2$, which is a torus in $\mathbb{S}^3 \setminus K$. If $\mathbb{T}$ is embedded, then $\mathbb{T}$ cannot be boundary-parallel in $\mathbb{S}^3 \setminus K$.
\end{lemma}

\begin{proof}
    We argue by contradiction. Assume that $\mathbb{T}$ is boundary-parallel. Cutting $\mathbb{S}^3$ along $\mathbb{T}$ yields two components, $X_1$ and $X_2$. Without loss of generality, we may assume that the knot $K$ lies entirely in $X_1$, which is a solid torus. Let $R$ denote the connected component of $S_2 \cap X_1$ that contains $K$.

    The inclusion $i: R \hookrightarrow X_1$ induces a homomorphism $i_* \colon H_1(R) \longrightarrow H_1(X_1)$. Since $K$ is the core of $X_1$, the class $i_*([K])$ generates $H_1(X_1) \cong \mathbb{Z}$. Moreover, since $[K] = [R \cap \mathbb{T}] \in H_1(R)$, it follows that $i_*([R \cap \mathbb{T}])$ also generates $H_1(X_1)$.
    
    Because the $1$-manifold $R \cap \mathbb{T}$ lies on the annulus $A_1$ and $C_1$, $C_2$ cobound $A_1$, the homology class $i_*([R \cap A_1])$ must be a multiple of $[C_1] = [C_2]$ in $H_1(X_1)$. Combined with the fact that $i_*([R \cap \mathbb{T}])$ generates $H_1(X_1)$, this implies that both $[C_1]$ and $[C_2]$ are equal to $i_*([R \cap A_1])$ and they generate $H_1(X_1)$. Consequently, the pair $(X_1, A_1)$ admits a product structure $(X_1, A_1) \cong \left(A_1 \times [0,1],\, A_1 \times \{0\}\right)$.

    Then we will prove that $R$ is an annulus. Since $S_2$ is incompressible and $R$ is connected and lies in the solid torus $X_1$, the surface $R$ can have neither positive genus nor more than two boundary components. This is because $\pi_1(X_1 \setminus K) \cong \mathbb{Z}^2$ is abelian. If $R$ had positive genus or more than two boundary components, then $\pi_1(R \setminus K)$ would be non-abelian. Consequently, the homomorphism $i_*$ on fundamental groups induced by the inclusion $i: R \setminus K \hookrightarrow X_1 \setminus K$ would not be injective. Hence, there exists a nontrivial element $[\gamma] \in \pi_1(R \setminus K)$ such that $i_*([\gamma]) = 0$. In other words, $\gamma$ is null-homotopic in $X_1 \setminus K$, and therefore also null-homotopic in $\mathbb{S}^3 \setminus K$, which contradicts the incompressibility of $S_2$.

    Moreover, since $R$ cannot be a disk (otherwise $R$ and $A_2$ would lie in different connected components of $S_2$, contradicting the connectedness of $S_2$), we conclude that $R$ is an annulus. This implies that $K$ and $C_3 := R \cap A_1 \subset S_1 \cap S_2$ are isotopic on $S_2$, which contradicts the fact that except for $K$, every intersection loop on $S_2$ is not boundary-parallel (see \cref{defi: Favourable pair}).
\end{proof}

\begin{remark}
    In \cref{lem: not boundary-parallel}, the assumption that $S_1$ has minimal genus can in fact be removed (although this assumption appears in the definition of a favourable pair), since it is not used anywhere in the proof.
\end{remark}

\section{Infinite Cyclic Cover} \label{Diameter Kakimizu section: Infinite Cyclic Cover}

\subsection{Definition of infinite cyclic cover}

\begin{definition}[Infinite cyclic cover]\label{infinitecycliccover}
    Let $K \subset \mathbb{S}^3$ be a knot with open tubular neighbourhood $W(K)$. The infinite cyclic cover $\widetilde{E}$ of the knot exterior $E = \mathbb{S}^3 \setminus W(K)$ is the covering space associated to the kernel of the abelianisation $\mathrm{ab}:\pi_1(E) \rightarrow H_1(E) \cong \mathbb{Z}$.
\end{definition}

\begin{definition}[Deck transformation]\label{Decktransformation}
    Let $p: \widetilde{E} \rightarrow E$ be a covering map. The \emph{group of deck transformations} $\operatorname{Deck}(\widetilde{E}/E)$ is the group of homeomorphisms $\varphi: \widetilde{E} \rightarrow \widetilde{E}$ such that $p \circ \varphi = p$. By \cref{infinitecycliccover}, the group $\operatorname{Deck}(\widetilde{E}/E)$ is isomorphic to $\mathbb{Z}$. We denote by $\tau: \widetilde{E} \rightarrow \widetilde{E}$ a generator of $\operatorname{Deck}(\widetilde{E}/E)$.
\end{definition}

\begin{definition}[Lifted Seifert surfaces]\label{lift}
    Let $S$ be a Seifert surface for $K$. From now on, when no confusion arises, we do not distinguish between $S$ and $S \setminus W(K)$, and we denote both simply by $S$. When lifted to the infinite cyclic cover $\widetilde{E}$, we pick one lift as $(\widetilde{S})_0$, and for any integer $i\in\mathbb{Z}$, define $(\widetilde{S})_i = \tau^i((\widetilde{S})_0)$, where $\tau$ is the generator of the deck transformation group. So $(\widetilde{S})_i:i\in\mathbb{Z}$ are all lifts of $S$.

    The collection of surfaces $\{(\widetilde{S})_i\}_{i\in\mathbb{Z}}$ divides $\widetilde{E}$ into countably many $3$-dimensional regions $\{\widetilde{E}_i(S)\}_{i\in\mathbb{Z}}$, where each $\widetilde{E}_i(S)$ is divided by $(\widetilde{S})_i$ and $(\widetilde{S})_{i+1}$. (See \cref{fig: lifted-surfaces} for illustration.)
\end{definition}

\subsection{Infinite cyclic cover decomposition}

\begin{definition}\label{intersection_seifert_surface}
    Let $K$ be a knot in $\mathbb{S}^3$, and let $S_1$ and $S_2$ be two distinct Seifert surfaces for $K$ such that $S_1 \cap S_2$ is a disjoint union of simple loops. For any integers $i$ and $j$, define $\widetilde{C}_{i,j} := (\widetilde{S_1})_i \cap (\widetilde{S_2})_j$ which we call the \emph{set of intersection loops}. By the connectedness of $(\widetilde{S_2})_0$ and the fact that the number of components of the intersection of $S_1$ and $S_2$ is finite, there exists a finite sequence of \emph{consecutive} integers $\{ h \in \mathbb{Z} \mid \widetilde{C}_{h,0} \neq \emptyset \}$, or this set is empty if $S_1$ and $S_2$ are disjoint.
\end{definition}

  \begin{figure}[H]
  \centering
  \includegraphics[width=0.9\textwidth]{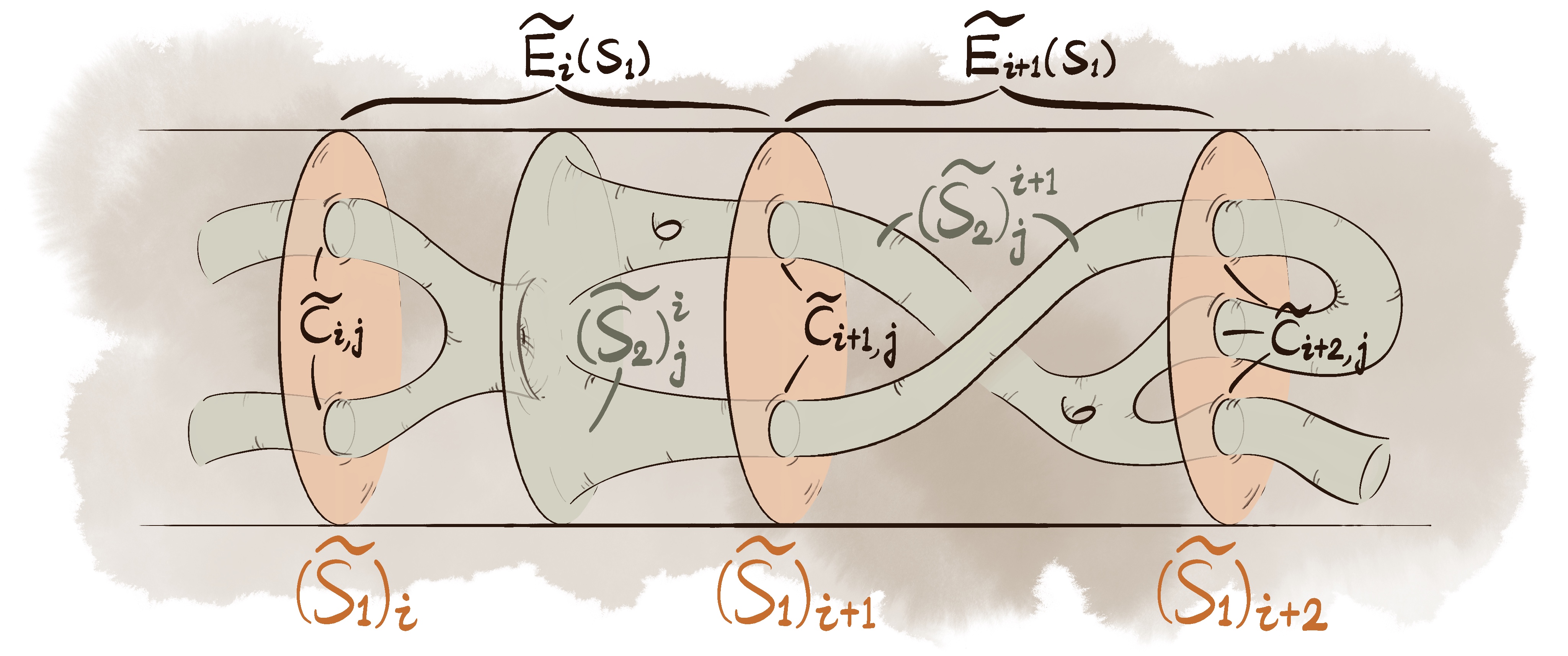}
  \caption{A lift of the Seifert surface $S_2$.}
  \label{fig: lifted-surfaces}
  \end{figure}

\begin{definition}\label{intersection_seifert_surface2}
    Let $(\widetilde{S_1})_i^j \subset (\widetilde{S_1})_i$ be defined as $(\widetilde{S_1})_i \cap \widetilde{E}_j(S_2)$, which is a subsurface of the surface $(\widetilde{S_1})_i$, and similarly, let $(\widetilde{S_2})_j^i \subset (\widetilde{S_2})_j$ be defined as $(\widetilde{S_2})_j \cap \widetilde{E}_i(S_1)$.
\end{definition}

\begin{definition}
Since $S_2$ is a Seifert surface and hence connected, the set $\left\{\, h \in \mathbb{Z} \;\middle|\; (\widetilde{S_2})_0^h \neq \emptyset \,\right\}$is a consecutive interval in $\mathbb{Z}$. We denote this interval by $\{\, h_1, h_1+1, \cdots, h_2 \,\}$.
\end{definition}

\begin{proposition}
    $S_1 \cap S_2 = \bigcup_{h=h_{1}}^{h_{2}} p \left( \widetilde{C}_{h,0} \right)$.
\end{proposition}

\begin{proposition}
    \label{pro: shift}
    $p(\widetilde{C}_{i,j}) = p(\widetilde{C}_{i+k,j+k}), \forall k \in \mathbb{Z}$.
\end{proposition}

\begin{proof}
    By definition, $\widetilde{C}_{i+k,j+k} = \tau^{k}(\widetilde{C}_{i,j})$. The result follows.
\end{proof}

\begin{corollary}
    $\left\{h \in \mathbb{Z} \mid \widetilde{C}_{0,h} \neq \emptyset \right\} = \left\{-h \mid h \in \mathbb{Z}, \widetilde{C}_{h,0} \neq \emptyset \right\}$.
\end{corollary}

\begin{proposition}\label{partial}
    $\partial(\widetilde{S_1})_i^j = \widetilde{C}_{i,j} \cup \widetilde{C}_{i,j+1}$ and $\partial(\widetilde{S_2})_j^i = \widetilde{C}_{i,j} \cup \widetilde{C}_{i+1,j}$.
\end{proposition}

\begin{proposition}\label{union}
   $\bigcup_j (\widetilde{S_1})_i^j = (\widetilde{S_1})_i$ and $\bigcup_i (\widetilde{S_2})_j^i = (\widetilde{S_2})_j$.
\end{proposition}

\begin{figure}[H]
\centering
\includegraphics[width=0.9\textwidth]{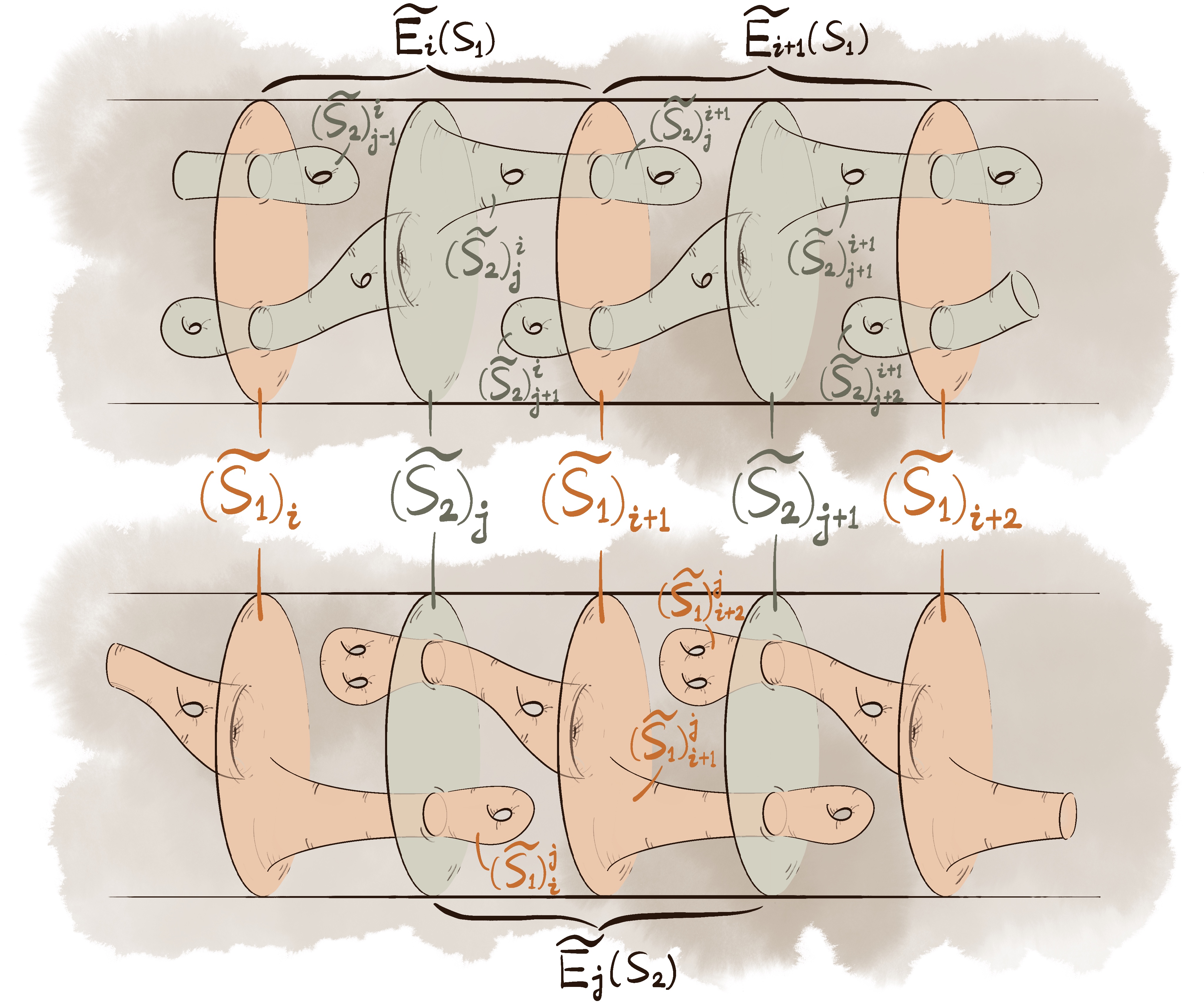}
\caption{The relative position of the lifts of $S_1$ and $S_2$.}
\label{fig: corresponding}
\end{figure}

\subsection{Kakimizu's characterisation}

\begin{definition}[Kakimizu's characterisation, \cite{MR2531146}]
    \label{defi: Kakimizu distance}
    Consider two Seifert surfaces $S_1$ and $S_2$. Define
    \begin{align}
        d_*(S_1,S_2) := \# \left\{ h \in \mathbb{Z} \;\middle|\; (\widetilde{S_2})_0^h \neq \emptyset \right\} \notag
    \end{align}
    if $[S_1] \neq [S_2]$, and define $d_*(S_1,S_2) := 0$ if $[S_1] = [S_2]$. 
    
    For two isotopy classes $[S_1]$ and $[S_2]$, we further define
    \begin{align}
        d_*([S_1],[S_2]) := \min_{S_i \in [S_i],i=1,2} d_*(S_1,S_2). \notag
    \end{align}
\end{definition}

\begin{lemma}[{\cite[Proposition 1.4(iii)]{MR1177053}}]
\label{lem: Kakimizu}
    For any Seifert surfaces $S_1, S_2, S_3$, we have
    \begin{align}
        d_*(S_1,S_3) \le d_*(S_1,S_2) + d_*(S_2,S_3). \notag
    \end{align}
\end{lemma}


\begin{theorem}
    \label{kakimizu1}
    \begin{enumerate}
        \item (\cite[Proposition 3.1]{MR1177053}) For any two vertices $[S_1]$ and $[S_2]$ with both genus at most $\ell$, we have
    \begin{align}
        d_\infty([S_1],[S_2]) = d_*([S_1],[S_2]).
    \end{align}
        \item Moreover, if $S_1$ has minimal genus, then we have
    \begin{align}
        d_\ell([S_1],[S_2]) = d_\infty([S_1],[S_2]) = d_*([S_1],[S_2]).
    \end{align}
    \end{enumerate}
\end{theorem}

\begin{proof}
    Kakimizu \cite{MR1177053} proved the first assertion of this theorem, while the proof of the second assertion essentially follows Kakimizu's original argument \cite{MR1177053}. It is straightforward to extend this argument to the setting of $IS_\ell(K)$ provided that $S_1$ has minimal genus. Indeed, Kakimizu constructed a path in $IS_{\infty}(K)$ connecting $[S_1]$ and $[S_2]$. Since $S_1$ has minimal genus, every Seifert surface appearing along this path has genus at most $\ell$, and hence the same path lies entirely in $IS_\ell(K)$. For completeness, we present this proof below to make this point explicit.
    
    In the following proof, we shall simply write $d$ to denote either $d_\ell$ or $d_\infty$. In the following, we denote $d_* = d_*([S_1],[S_2])$ and $d = d([S_1],[S_2])$, and assume that the surfaces $S_1, S_2$ satisfy $d_*(S_1,S_2) = d_*([S_1],[S_2])$. 
    
    \textbf{First, we prove that $d_*\leq d$.} Since $d([S_1],[S_2])=d$, we may take a sequence of Seifert surfaces $S_1 = R_0, R_1, \cdots, R_d = S_2$ such that for each $0 \leq i \leq d-1$, $R_i$ and $R_{i+1}$ are disjoint in $S^3 \setminus K$. Since $d(R_i,R_{i+1}) = 1$, the surfaces $R_i$ and $R_{i+1}$ are disjoint, and hence $d_*(R_i,R_{i+1}) = 1$. Therefore, we obtain    
        \begin{align}
        d &= d(R_0,R_{1})+\cdots+d(R_{d-1},R_{d})=1+\cdots+1 \notag\\
        &= d_*(R_0,R_{1})+\cdots+d_*(R_{d-1},R_{d})\geq d_*(R_0,R_{d}) = d_* , \notag
        \end{align}
    where the last inequality follows from \cref{lem: Kakimizu}.


    \textbf{Second, we prove that $d_* \geq d$.}
    \begin{itemize}
    \item We proceed by induction on $d_*$. If $d_* = 1$, then the surfaces $S_1$ and $S_2$ are disjoint, and hence clearly $d = 1$. Now assume $d_* \geq 2$.
    
    \item We construct a Seifert surface $S_2'$ such that $d_*(S_1, S_2') \leq d_* - 1$ and $d_*(S_2, S_2') = 1$, as follows.
    
    Assume that $\left\{ h \in \mathbb{Z} \;\middle|\; (\widetilde{S_2})_0^h \neq \emptyset \right\} = \{h_1,h_1+1,\cdots,h_2\}$, that is, $(\widetilde{S_2})_0$ intersects $\widetilde{E}_{h_1}(S_1), \cdots, \widetilde{E}_{h_2}(S_1)$. Let $\partial (\widetilde{S_2})_0 \subset \widetilde{E}_{k}(S_1)$ for some $h \leq k \leq h_2$, where $h_1 \neq h_2$ (since $d_* \geq 2$). If $k < h_2$, define
    \begin{align}
        (\widetilde{S_2''})_0 := \left( (\widetilde{S_2})_0 \setminus \widetilde{E}_{h_2}(S_1) \right) \cup \left( (\widetilde{S_1})_{h_2} \cap \widetilde{E}_{-1}(S_2) \right). \notag
    \end{align}
    Otherwise, if $k = h_2$, then $k > h_1$, we may define
    \begin{align}
        (\widetilde{S_2''})_0 := \left( (\widetilde{S_2})_0 \setminus \widetilde{E}_{h_1}(S_1) \right) \cup \left( (\widetilde{S_1})_{h_1+1} \cap \widetilde{E}_{0}(S_2) \right). \notag
    \end{align}
    Then, after a small perturbation isotopy, $(\widetilde{S_2''})_0$ intersects $\widetilde{E}_{h_1}(S_1), \cdots, \widetilde{E}_{h_2-1}(S_1)$ if $k < h_2$ (or $\widetilde{E}_{h_1+1}(S_1), \cdots, \widetilde{E}_{h_2}(S_1)$ if $k = h_2$). 
    
    Let $S_2'' = p\bigl( (\widetilde{S_2''})_0 \bigr)$, where $p$ denotes the covering projection. Since $(\widetilde{S_2''})_0$ can be isotoped into $\widetilde{E}_{-1}(S_2)$ if $k < h_2$ (or $\widetilde{E}_{0}(S_2)$ if $k = h_2$), the surfaces $S_2$ and $S_2''$ are disjoint. Hence $d_*(S_2, S_2'') = 1 = d(S_2, S_2'')$.

    \item Note that we cannot directly assert that $S_2''$ is incompressible. In what follows, we perform a sequence of operations to modify $S_2''$ into an incompressible Seifert surface $S_2'$. If $S_2''$ is compressible, then there exists a loop $\delta \subset S_2''$ that is not homotopically trivial on $S_2''$ and bounds a disk $D$ in $\mathbb{S}^3 \setminus S_2''$.

    We then cut $S_2''$ along $\delta$ and cap off the two resulting boundary components with copies of $D$, producing a surface that may be disconnected. Taking the connected component containing the boundary $K$, we obtain a Seifert surface, which we denote by $S_2'''$, whose genus is reduced by at least one.

    Observe that we may adjust the disk $D$ so that $S_2'''$ is disjoint from $S_2$, i.e.,
    \begin{align}
        d_{*}(S_2,S'''_2) = 1. \notag
    \end{align}
    Indeed, suppose that $D$ intersects $S_2$ transversely. Since $S_2$ and $S_2''$ are disjoint, we have $\partial D \cap S_2 = \emptyset$. Hence, the intersection $D \cap S_2$ consists of a collection of simple loops, each of which is homotopically trivial on $D$, and therefore trivial in $\mathbb{S}^3 \setminus K$. Since $S_2$ is incompressible, these loops are also homotopically trivial on $S_2$. By an isotopy of $D$, we may eliminate all such intersection loops and hence assume that $D$ is disjoint from $S_2$.

    Moreover, we may further adjust the disk $D$ so that we can arrange
    \begin{align}
        (\widetilde{S_2'''})_0 \subset \bigcup_{n=h_1}^{h_2-1}\widetilde{E}_{n}(S_1) \notag
    \end{align}
    if $k < h_2$ (the case $k = h_2$ can be treated analogously, so in what follows we always assume $k < h_2$). Consequently, we can always ensure that
    \begin{align}
        d_*(S_1,S_2''') \leq d_* - 1. \notag
    \end{align}
    Indeed, if $(\widetilde{S_2'''})_0$ were to leave these regions (without loss of generality, suppose that $(\widetilde{S_2'''})_0$ intersects $(\widetilde{S_1})_{h_2}$), then, since $(\widetilde{S_2''})_0$ is disjoint from $(\widetilde{S_1})_{h_2}$, the intersection must occur in the interior of $\widetilde{D}$. As in the previous argument, this would imply that $(\widetilde{S_1})_{h_2}$ is compressible, which is a contradiction.
    
    Repeating the above procedure of cutting along a compressing loop and capping off with appropriate disks until the surface becomes incompressible yields a Seifert surface, denoted by $S_2'$, which satisfies $d_*(S_2,S_2') = 1$ and $d_*(S_1,S_2') \leq d_* - 1$.

    \item For the proof in the case of $IS_\ell(K)$, one final step remains: using the condition that $S_1$ has minimal genus, we show that the genus of $S_2'$ is at most $\ell$. In fact, it suffices to prove that the genus of $S_2''$ is at most $\ell$.
    
    Suppose, for contradiction, that this is not the case. Then, in the case $k < h_2$ (the case $k = h_2$ is analogous), the Euler characteristic of $(\widetilde{S_2})_0 \cap \widetilde{E}_{h_2}(S_1)$ would be strictly larger than that of $(\widetilde{S_1})_{h_2} \cap \widetilde{E}_{-1}(S_2)$. Consider the surface
    \begin{align}
        \widetilde{S_1'} := \left( (\widetilde{S_2})_0 \cap \widetilde{E}_{h_2}(S_1) \right) \cup \left( (\widetilde{S_1})_{h_2} \setminus \widetilde{E}_{-1}(S_2) \right). \notag
    \end{align}
    It follows that $p(\widetilde{S_1'})$ is a Seifert surface whose genus is strictly smaller than that of $S_1$, contradicting the assumption that $S_1$ has minimal genus.
    
    \item By the induction hypothesis, we have
    \begin{align}
        d_* - 1 \geq d_*(S_1, S_2') \geq d(S_1, S_2'). \notag
    \end{align}
    Adding $1$ to both sides yields
    \begin{align}
        d_* \geq d(S_1, S_2') + 1 = d(S_1, S_2') + d(S_2, S_2') \geq d. \notag
    \end{align}

    \end{itemize}

    Combining the above arguments, we have shown that $d_* \leq d$ and $d_* \geq d$. Therefore, the lemma follows.
    
\end{proof}

\begin{proposition}
    Let $K$ be a knot of genus $g$. Then for any $\ell \geq g$, the complex $IS_{\ell}(K)$ is connected.
\end{proposition}

\begin{proof}
    Take any two vertices $[S_1],[S_2] \in IS_{\ell}(K)$. Let $[S_3] \in IS_{g}(K) \subset IS_{\ell}(K)$ be an arbitrary minimal-genus Seifert surface. By the second conclusion of \cref{kakimizu1}, there exists a path in $IS_{\ell}(K)$ connecting $[S_1]$ and $[S_3]$. Similarly, there exists a path in $IS_{\ell}(K)$ connecting $[S_2]$ and $[S_3]$. Therefore, any two vertices $[S_1]$ and $[S_2]$ lie in the same connected component of $IS_{\ell}(K)$. Hence, $IS_{\ell}(K)$ is (path-)connected.

    This result is included in \cref{main thm2: Hyperbolic Knot Linear Bound}.
\end{proof}

\begin{remark}
    \label{kakimizu2}
    By \cref{kakimizu1}, we immediately obtain the following: for any two vertices $[S_1]$ and $[S_2]$ in $IS(K)$ of minimal genus $g=g(K)$, we have $d_g([S_1],[S_2]) = d_\infty([S_1],[S_2]) = d_*([S_1],[S_2])$. In fact, this result was already proved by Kakimizu in \cite[Proposition 3.1]{MR1177053}.
\end{remark}


\section{Annulus Surgery}

To prepare for the proof of the main theorem \cref{thm:euler-inequality} in the next section, we introduce the concept of annulus surgery.

\begin{definition}[$I$-shaped and $U$-shaped Annuli]
    \label{defi: I and U}

Let $K$ be a knot in $\mathbb{S}^3$, and let $S_1$ and $S_2$ be two distinct Seifert surfaces for $K$ such that $S_1 \cap S_2 \neq \emptyset$ is a disjoint union of simple loops. For positive integers $i,j$, suppose that $(\widetilde{S_2})_{j}^{i}$ contains an annular
connected component, which we denote by $\widetilde{A}$. In what follows, we only consider such annuli $\widetilde{A}$ whose two boundary components lie entirely in $\widetilde{C}_{i,j} \cup \widetilde{C}_{i+1,j}$, and do not lie in $\partial \widetilde{E}$.

\begin{itemize}
    \item If the two boundary components of $\widetilde{A}$ lie in \emph{distinct} sets $\widetilde{C}_{i,j}$ and $\widetilde{C}_{i+1,j}$, we call it an \emph{$I$-shaped annulus}.
    \item Otherwise, if both boundary components lie in the \emph{same} set, we call it a \emph{$U$-shaped annulus} (see \cref{fig: I-shaped and U-shaped}).
\end{itemize}
    Similarly, we can define $I$-shaped and $U$-shaped annuli in $(\widetilde{S_1})_{i}^{j}$.
\end{definition}

\begin{figure}[H]
  \centering
  \includegraphics[width=0.9\textwidth]{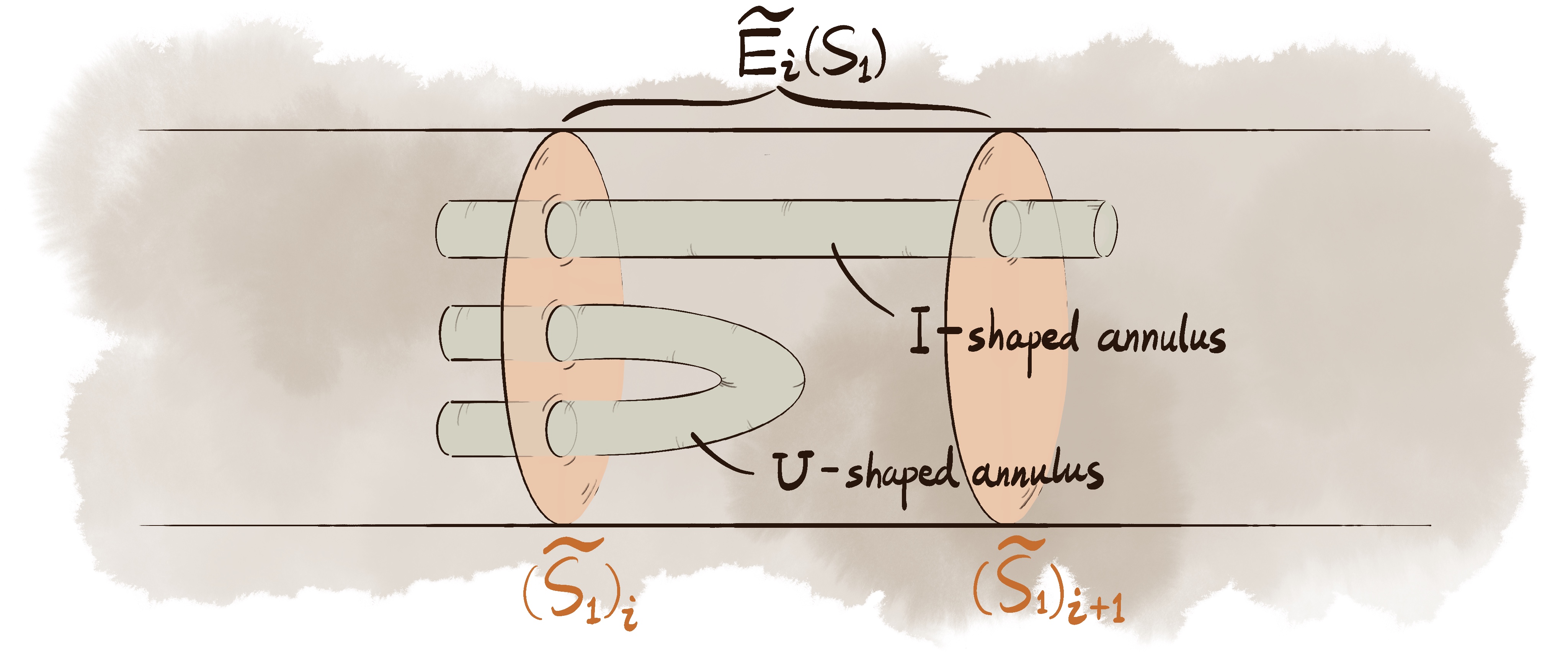}
  \caption{$I$-shaped and $U$-shaped annuli.}
  \label{fig: I-shaped and U-shaped}
  \end{figure}

The proof strategy of \cref{thm:euler-inequality} fundamentally relies on the annulus surgery technique introduced later.

\begin{definition}[Annulus Surgery]
\label{annulussurgery}
Let $A \subset S_2$ be an annulus such that its lift $\widetilde{A}$ is a $U$-shaped annulus. 
Consider a thickening $A \times I$ of $A$ such that $(\partial A) \times I \subset S_1$. 
We perform the following two steps:
\begin{itemize}
    \item Modify $S_1$ to obtain $S_1'$ by
    \begin{align}
        S_1' := \left( S_1 \setminus \left( (\partial A) \times I \right) \right) \cup \left( A \times \{0,1\} \right). \notag
    \end{align}
    \item We take the connected component of $S_1'$ containing the boundary $K$ and denote it by $s(S_1)$.
\end{itemize}
We say that $s(S_1)$ is obtained from $S_1$ by an \emph{annulus surgery} along $A$ on $S_2$.
\end{definition}

\begin{figure}[H]
\centering
\includegraphics[width=0.9\textwidth]{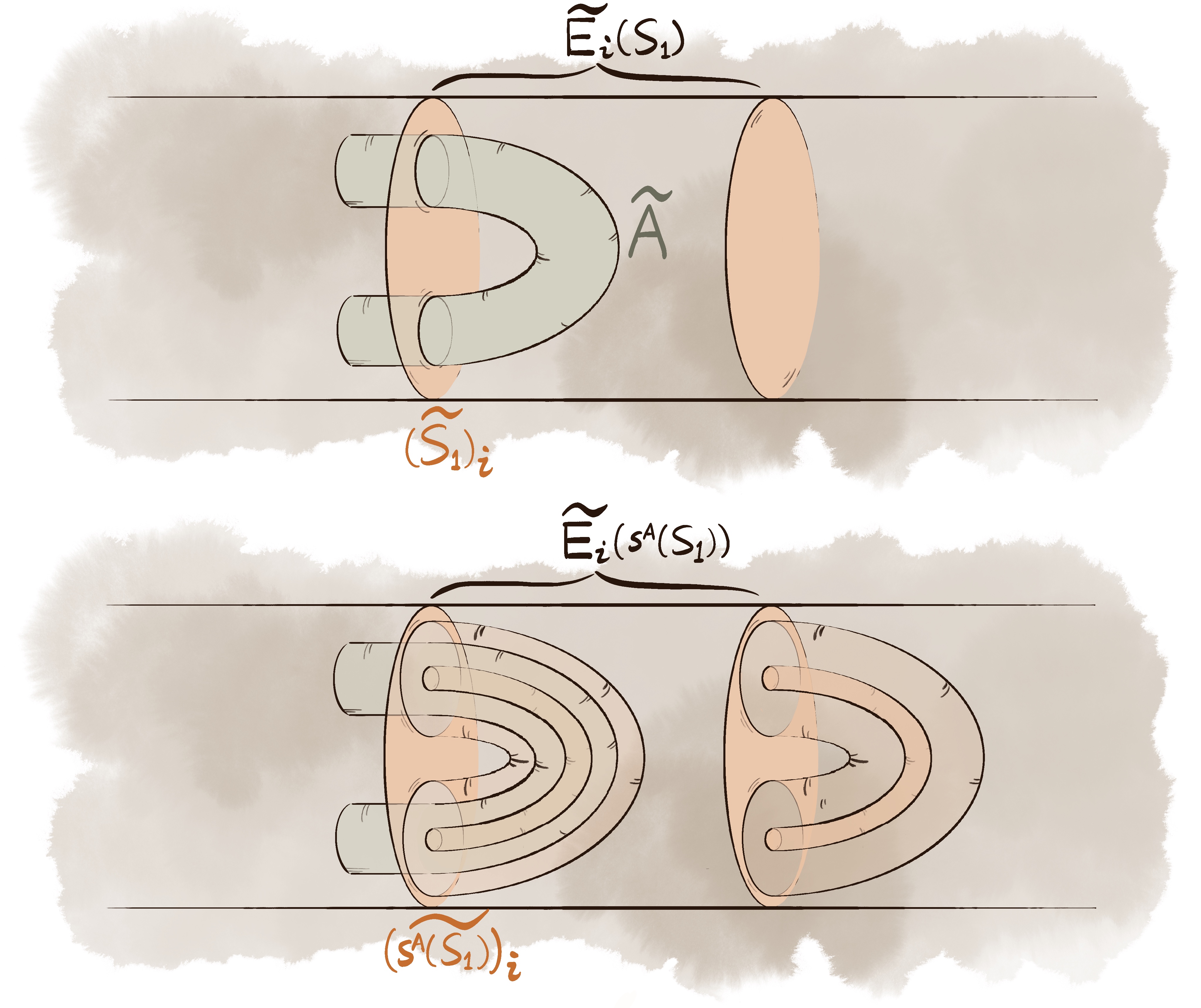}
\caption{Annulus surgery.}
\label{fig: Annulus surgery}
\end{figure}

\begin{remark}
    For an intuitive understanding, we may describe the annulus surgery in the infinite cyclic cover (see \cref{fig: Annulus surgery}).

    This figure also explains why annulus surgery can only be performed on $U$-shaped annuli, but not on $I$-shaped annuli. Since $\widetilde{A}$ lies entirely on one side of $\widetilde{S_1}$, the resulting surface after the surgery is orientable. If one were to perform the surgery on an $I$-shaped annulus instead, the resulting surface would be non-orientable.
\end{remark}

\begin{remark}
    The definition of annulus surgery is not restricted to being performed along a $U$-shaped annulus $A \subset S_2$, but can also be extended to a $U$-shaped annulus $A \subset S_1$. For convenience of notation, we view annulus surgery as a transformation from a pair of surfaces to another pair of surfaces:
    \begin{align}
        s(S_1,S_2) = (s^{A}(S_1), s^{A}(S_2)). \notag
    \end{align}
    Here, the superscript $A$ indicates that the surgery is performed along the annulus $A$, which may lie in either $S_1$ or $S_2$. If $A \subset S_2$, we define $s^{A}(S_2) := S_2$, whereas if $A \subset S_1$, we define $s^{A}(S_1) := S_1$. When no confusion arises, we will sometimes omit the superscript.

    Moreover, we denote
    \begin{align}
        \widetilde{(s^{A}(S_1))}_i^j := \widetilde{E}_j(s^{A}(S_2)) \cap (\widetilde{s^{A}(S_1)})_i
        \quad \text{and} \quad
        \widetilde{(s^{A}(S_2))}_j^i := \widetilde{E}_i(s^{A}(S_1)) \cap (\widetilde{s^{A}(S_2)})_j .\notag
    \end{align}
\end{remark}

\begin{proposition} \label{pro:minimal-genus}
Let $A$ be an annulus on $S_2$ such that its lift $\widetilde{A}$ is a $U$-shaped annulus. Then,
\begin{align}
        g(s(S_1)) \leq g(S_1).\notag
    \end{align}
In particular, if $S_1$ has minimal genus, then $g(s(S_1)) = g(S_1)$; that is, $s(S_1)$ also has minimal genus.
\end{proposition}

\begin{proof}
Observe that $S'_1$ is obtained from $S_1$ by cutting along two disjoint closed loops (creating four boundary components) and then pairwise regluing these boundary components. This process preserves the Euler characteristic:
\begin{equation}
    \chi(S_1) = \chi(S'_1). \notag
\end{equation}
Moreover, since $\widetilde{A}$ lies entirely on one side of $\widetilde{S_1}$, the resulting surface $S'_1$ after the surgery is orientable.
\begin{itemize}
    \item Connected case: if $S'_1$ is connected, then $s(S_1) = S'_1$. Since $\chi(S_1) = \chi(S'_1)$ and both surfaces have only one boundary component, namely $K$, it follows that $s(S_1)$ and $S_1$ have the same genus.

    \item Disconnected case: if $S'_1$ is disconnected, then it consists of two components: the surface $s(S_1)$ bounded by $K$, and a closed surface $S^{''}_1 := S'_1 \setminus s(S_1)$. The closed component $S^{''}_1$ cannot be a sphere, since it contains at least one handle in its topology. Consequently, $g(S^{''}_1) \geq 1$. Therefore, by the Euler characteristic relation,
    \begin{align}
        \chi(S_1) = \chi(S'_1) = \chi(s(S_1)) + \chi(S^{''}_1) \leq \chi(s(S_1)), \notag
    \end{align}
    which implies that $g(S_1) \geq g(s(S_1))$. In particular, if $S_1$ has minimal genus, then the genus of $s(S_1)$ cannot be smaller than that of $S_1$, and hence $g(S_1) = g(s(S_1))$. That is, both surfaces have minimal genus.
\end{itemize}
\end{proof}

\begin{figure}[H]
\centering
\includegraphics[width=0.82\textwidth]{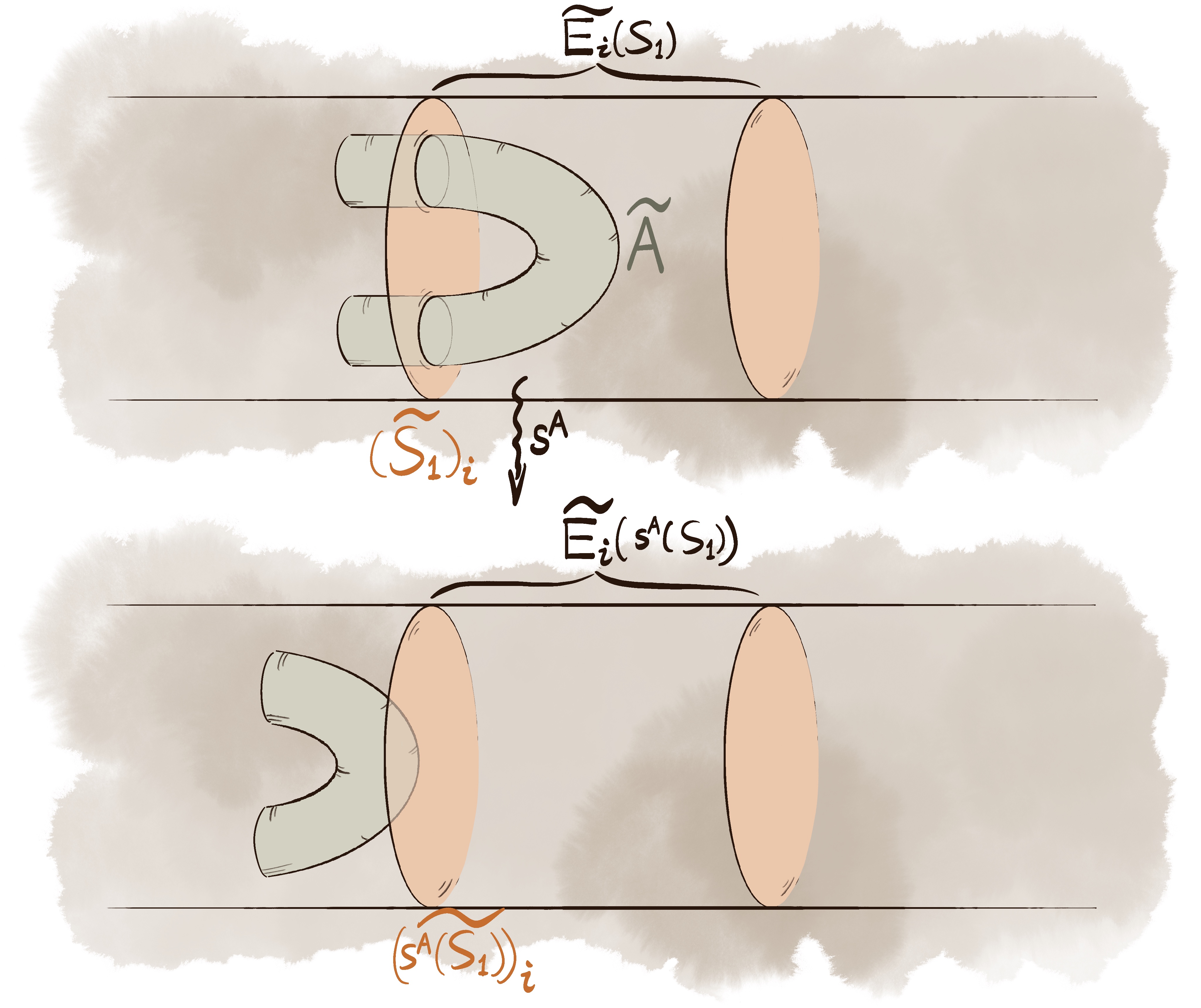}
\caption{The standard point of view of $S_1$ and $\mathcal{S}_1$.}
\label{fig: view of S1}
\end{figure}

\begin{figure}[H]
\centering
\includegraphics[width=0.82\textwidth]{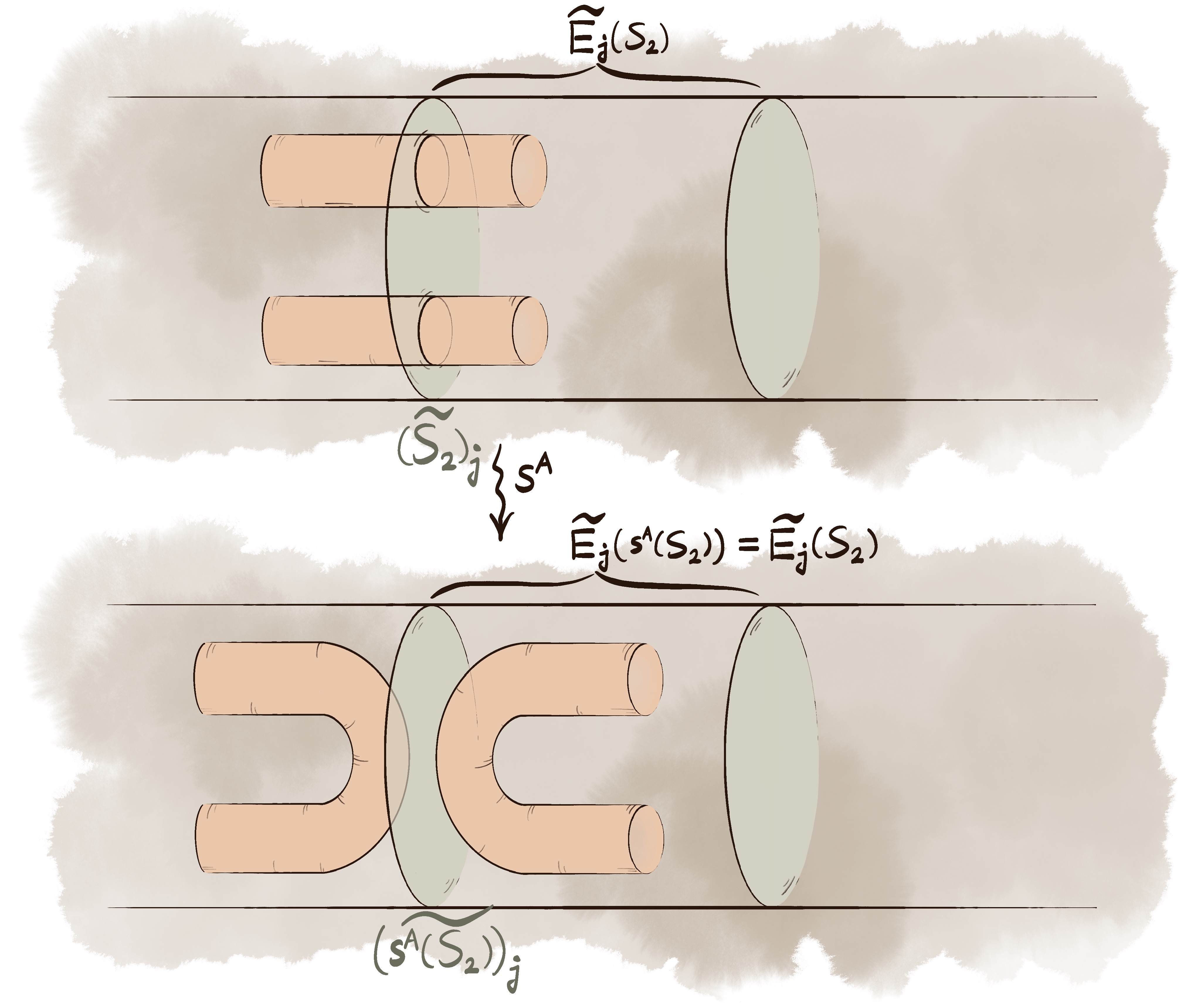}
\caption{The standard point of view of $S_2$.}
\label{fig: view of S2}
\end{figure}

By performing annulus surgery along $A \subset S_2$, we transform the pair $(S_1, S_2)$ into $s(S_1,S_2)$. If we view $(\widetilde{S_2})_j$ with respect to $\{(\widetilde{S_1})_i\}_{i\in\mathbb{Z}}$ before the surgery, and view $(\widetilde{s(S_2)})_j = (\widetilde{S_2})_j$ with respect to $\{(\widetilde{s(S_1)})_i\}_{i\in\mathbb{Z}}$ after the surgery, then the effect of the surgery can be described as follows: the annulus $\widetilde{A}$ is pushed out of the fundamental region (see \cref{fig: view of S1}).

On the other hand, if we view $(\widetilde{S_1})_i$ with respect to $\{(\widetilde{S_2})_j\}_{j\in\mathbb{Z}}$ before the surgery, and view $(\widetilde{s(S_1)})_i$ with respect to $\{(\widetilde{s(S_2)})_j\}_{j\in\mathbb{Z}} = \{(\widetilde{S_2})_j\}_{j\in\mathbb{Z}}$ after the surgery, then the surgery may be regarded as cutting the surface $(\widetilde{S_1})_i$ along two loops and then regluing them in such a way that the resulting surface locally lies in different fundamental regions (if the surgery produces disconnected components, we further discard the closed component) (see \cref{fig: view of S2}).

Based on the above observations, in particular, we obtain the following property.

\begin{proposition}
\label{prop: chi does not change by surgery}
Let $(s(S_1), s(S_2))$ be the pair of Seifert surfaces obtained from $(S_1, S_2)$ by annulus surgery along $A \subset S_2$. Recall that $\widetilde{(s(S_1))}_i^j := \widetilde{E}_j(s(S_2)) \cap (\widetilde{s(S_1)})_i$ and $\widetilde{(s(S_2))}_j^i := \widetilde{E}_i(s(S_1)) \cap (\widetilde{s(S_2)})_j$. Then for any $i,j \in \mathbb{Z}$, we have
\begin{align}
\chi\bigl((\widetilde{s(S_2)})_j^i\bigr) = \chi\bigl((\widetilde{S_2})_j^i\bigr)
\end{align}
and
\begin{align}
\chi\bigl((\widetilde{s(S_1)})_i^j\bigr) \geq \chi\bigl((\widetilde{S_1})_i^j\bigr).
\end{align}
Moreover, if $S_1$ has minimal genus, then the above inequality is in fact an equality.
\end{proposition}

\begin{proof}
If we view $(\widetilde{S_2})_j$ with respect to $\{(\widetilde{S_1})_i\}_{i\in\mathbb{Z}}$ before the surgery, and view $(\widetilde{s(S_2)})_j$ with respect to $\{(\widetilde{s(S_1)})_i\}_{i\in\mathbb{Z}}$ after the surgery, then the effect of the surgery can be described as follows: the annulus $\widetilde{A} \subset (\widetilde{S_2})_j^i$ is pushed out of the fundamental region. Consequently, we have $\chi\bigl((\widetilde{s(S_2)})_j^i\bigr) = \chi\bigl((\widetilde{S_2})_j^i\bigr)$.

On the other hand, if we view $(\widetilde{S_1})_i$ with respect to $\{(\widetilde{S_2})_j\}_{j\in\mathbb{Z}}$ before the surgery, and view $(\widetilde{s(S_1)})_i$ with respect to $\{(\widetilde{s(S_2)})_j\}_{j\in\mathbb{Z}} = \{(\widetilde{S_2})_j\}_{j\in\mathbb{Z}}$ after the surgery, then the surgery may be regarded as cutting the surface $(\widetilde{S_1})_i$ along two loops and then regluing them in such a way that the resulting surface locally lies in different fundamental regions. Up to this stage, the Euler characteristic is unchanged. If this operation produces disconnected components, we further discard any closed components; we obtain $\chi\bigl((\widetilde{s(S_1)})_i^j\bigr) \geq \chi\bigl((\widetilde{S_1})_i^j\bigr).$

Now consider the case where $S_1$ has minimal genus. If any closed component appears during the annulus surgery, then it must be a torus, whose Euler characteristic is zero. Otherwise, the retained component (which is a Seifert surface) would have strictly smaller genus than $S_1$, which contradicts the assumption that $S_1$ has minimal genus. Therefore, in this case, annulus surgery does not change the Euler characteristic, and we have $\chi\bigl((\widetilde{s(S_1)})_i^j\bigr) = \chi\bigl((\widetilde{S_1})_i^j\bigr)$.

\end{proof}

Annulus surgery can be performed not only along a $U$-shaped annulus $A \subset S_2$, but also along a $U$-shaped annulus $A \subset S_1$. By symmetry, the conclusions of \cref{pro:minimal-genus} and \cref{prop: chi does not change by surgery} remain valid in this case.

\begin{remark}
    Note that, although annulus surgery guarantees that the genus does not increase and the Euler characteristic does not decrease, a Seifert surface obtained by annulus surgery is not necessarily incompressible. However, this does not affect the subsequent argument.
\end{remark}

\section{Central Zone}
\label{Diameter Kakimizu section: Central Zone}

The main result of this section is to prove the following inequality concerning Euler characteristics:

\begin{theorem}\label{thm:euler-inequality}
Let $K \subset \mathbb{S}^3$ be an atoroidal knot, and let $(S_1,S_2)$ be a favourable pair. Denote $\left\{ h \in \mathbb{Z} \;\middle|\; (\widetilde{S_2})_0^h \neq \emptyset \right\} =: \{ h_1, h_1+1, \cdots, h_2 \}$. Then there exist integers $h_3$ and $h_4$ with $h_1\leq h_3\leq h_4\leq h_2$ such that for any positive integer $h$:
\begin{itemize}
\item If $h_3\leq h\leq h_4$, then $\widehat{\chi}(h) := \chi((\widetilde{S_1})_{0}^{-h-1}) + \chi((\widetilde{S_2})_{0}^{h}) < 0$ or $\left((\widetilde{S_1})_{0}^{-h-1} \cup (\widetilde{S_2})_{0}^{h} \right) \cap \partial \widetilde{E} \neq \emptyset$.
\item Otherwise, $\chi((\widetilde{S_1})_{0}^{-h-1}) = 0$.
\end{itemize}
\end{theorem}

The reason for introducing the sum $\widehat{\chi}(h)$ is the following. By \cref{pro: shift} and \cref{partial}, we have
\begin{align}
p\bigl(\partial(\widetilde{S_1})_0^{-h-1}\bigr)
&= p\bigl(\widetilde{C}_{0,-h-1} \cup \widetilde{C}_{0,-h}\bigr) = p\bigl(\widetilde{C}_{h+1,0} \cup \widetilde{C}_{h,0}\bigr) = p\bigl(\partial(\widetilde{S_2})_0^{h}\bigr)
\end{align}
for all integers $h$. Consequently, the surface $p\bigl((\widetilde{S_1})_0^{-h-1} \cup (\widetilde{S_2})_0^{h}\bigr)$ is either a closed surface, or a surface with one or two boundary components, each of which is the knot $K$. Its Euler characteristic is precisely $\widehat{\chi}(h)$.

The first conclusion of the theorem, $\left((\widetilde{S_1})_{0}^{-h-1} \cup (\widetilde{S_2})_{0}^{h} \right) \cap \partial \widetilde{E} \neq \emptyset$, means that either the boundary of $(\widetilde{S_1})_{0}$ is contained in $(\widetilde{S_1})_{0}^{-h-1}$, or the boundary of $(\widetilde{S_2})_{0}$ is contained in $(\widetilde{S_2})_{0}^{h}$.

We call the sequence $\{h_3, \cdots, h_4\}$ the \emph{central zone} of the pair $(S_1, S_2)$.
By definition, the Euler characteristic $\chi\bigl((\widetilde{S_1})_{0}^{-h-1}\bigr)$ can be negative only for $h \in \{h_3, \cdots, h_4\}$.
Outside the central zone, we necessarily have $\chi\bigl((\widetilde{S_1})_{0}^{-h-1}\bigr)=0$. In other words, all nontrivial topological complexity of $(\widetilde{S_1})_{0}$ is concentrated within the central zone.

\begin{proof}
The proof of \cref{thm:euler-inequality} begins by assuming that there exists an integer $h$ in the range $h_1 \leq h \leq h_2$ such that $\chi((\widetilde{S_1})_{0}^{-h-1}) = \chi((\widetilde{S_2})_{0}^{h}) = 0$ and $\left((\widetilde{S_1})_{0}^{-h-1} \cup (\widetilde{S_2})_{0}^{h} \right) \cap \partial \widetilde{E} = \emptyset$. We aim to show that either for every integer $h_1' < h$, $\chi((\widetilde{S_1})_{0}^{-h_1'-1}) = 0$, or for every integer $h_2' > h$, $\chi((\widetilde{S_1})_{0}^{-h_2'-1}) = 0$.

Since $\chi((\widetilde{S_1})_{0}^{-h-1}) = \chi((\widetilde{S_2})_{0}^{h}) = 0$, both $(\widetilde{S_1})_{0}^{-h-1}$ and $(\widetilde{S_2})_{0}^{h}$ must consist of finite collections of disjoint annuli. Furthermore, the boundary of each such annulus lies entirely in $(\widetilde{S_2})_{-h-1} \cup (\widetilde{S_2})_{-h}$ (resp., $(\widetilde{S_1})_{h} \cup (\widetilde{S_1})_{h+1}$), since $\left((\widetilde{S_1})_{0}^{-h-1} \cup (\widetilde{S_2})_{0}^{h} \right) \cap \partial \widetilde{E} = \emptyset$. Therefore, $(\widetilde{S_1})_{0}^{-h-1}$ and $(\widetilde{S_2})_{0}^{h}$ consist of finite collections of $U$-shaped annuli and $I$-shaped annuli.

Suppose that either $(\widetilde{S_1})_{0}^{-h-1}$ or $(\widetilde{S_2})_{0}^{h}$ contains at least one $U$-shaped annulus. By performing annulus surgery along such a $U$-shaped annulus $A_1$, we obtain $(\widetilde{s^{A_1}(S_1)})_{0}^{-h-1}$ and $(\widetilde{s^{A_1}(S_2)})_{0}^{h}$.

This surgical procedure has two crucial effects. First, it preserves or decreases the genus of the surface, as established in \cref{pro:minimal-genus}; more precisely, when restricted to each region, it preserves or increases the Euler characteristic, as established in \cref{prop: chi does not change by surgery}. Second, it reduces the number of boundary components of $(\widetilde{S_1})_{0}^{-h-1}$ and $(\widetilde{S_2})_{0}^{h}$ by exactly two, respectively.

Since annulus surgery preserves or increases the Euler characteristic, and the Euler characteristics of $(\widetilde{s^{A_1}(S_1)})_{0}^{-h-1}$ and $(\widetilde{s^{A_1}(S_2)})_{0}^{h}$ cannot increase any further, they must still be equal to $0$. Equivalently, each of them remains a disjoint union of annuli. We now check whether $(\widetilde{s^{A_1}(S_1)})_{0}^{-h-1}$ and $(\widetilde{s^{A_1}(S_2)})_{0}^{h}$ contain any $U$-shaped annuli. If so, we perform the surgery again on the corresponding surface, we obtain $(\widetilde{s^{A_2}s^{A_1}(S_1)})_{0}^{-h-1}$ and $(\widetilde{s^{A_2}s^{A_1}(S_2)})_{0}^{h}$, and so on. Since each surgery reduces the number of boundary components of $(\widetilde{S_1})_{0}^{-h-1}$ and $(\widetilde{S_2})_{0}^{h}$, after finitely many steps, both of $(\widetilde{s^{A_m}\cdots s^{A_2}s^{A_1}(S_1)})_{0}^{-h-1}$ and $(\widetilde{s^{A_m}\cdots s^{A_2}s^{A_1}(S_2)})_{0}^{h}$ will contain only $I$-shaped annuli (or may be empty) and no $U$-shaped ones (see \cref{fig: ManyAnuulusSurgerys} for an example).

\begin{figure}[H]
\centering
\includegraphics[width=1.0\textwidth]{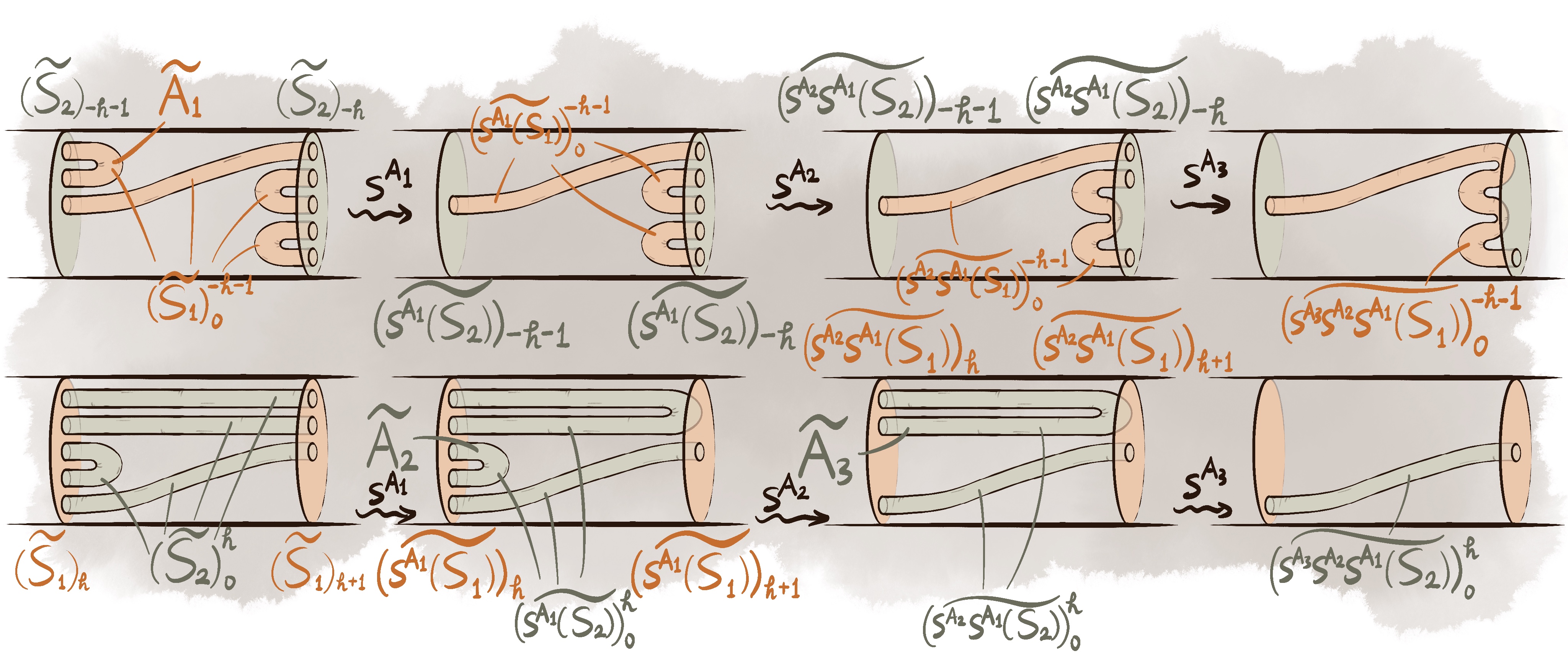}
\caption{Performing multiple annulus surgeries on $(\widetilde{S_1})_{0}^{-h-1}$ and $(\widetilde{S_2})_{0}^{h}$.}
\label{fig: ManyAnuulusSurgerys}
\end{figure}

We denote the final resulting surfaces by $(\widetilde{\mathcal{S}_1})_{0}^{-h-1}$ and $(\widetilde{\mathcal{S}_2})_{0}^{h}$. Analogous to Definition~\ref{intersection_seifert_surface}, we define the intersection curves $\widetilde{\mathcal{C}}_{i,j} := (\widetilde{\mathcal{S}_1})_i \cap (\widetilde{\mathcal{S}_2})_j$.

\begin{itemize}

    \item Case 1: both $\partial((\widetilde{\mathcal{S}_1})_{0}^{-h-1})$ and $\partial((\widetilde{\mathcal{S}_2})_{0}^{h})$ are empty; that is, both $(\widetilde{\mathcal{S}_1})_{0}^{-h-1}$ and $(\widetilde{\mathcal{S}_2})_{0}^{h}$ are empty.

    In this case, by the construction of the annulus surgeries, $(\widetilde{\mathcal{S}_1})_{0}$ is connected. Moreover, by \cref{pro:minimal-genus}, since $S_1$ has minimal genus, the surface $\mathcal{S}_1$ also has minimal genus. By \cref{prop: chi does not change by surgery} and the minimal-genus property of $S_1$, one of the following must hold: either for every integer $h_1' < h$, $\chi((\widetilde{S_1})_{0}^{-h_1'-1}) = \chi((\widetilde{\mathcal{S}_1})_{0}^{-h_1'-1}) = \chi(\emptyset) = 0$, or for every integer $h_2' > h$, $\chi((\widetilde{S_1})_{0}^{-h_2'-1}) = \chi((\widetilde{\mathcal{S}_1})_{0}^{-h_2'-1}) = \chi(\emptyset) = 0$. This completes the proof.

    Note that this step relies on the condition that $S_1$ has minimal genus. One cannot conclude that $\chi\!\bigl((\widetilde{S_2})_{0}^{h_1'}\bigr) = 0$ or $\chi\!\bigl((\widetilde{S_2})_{0}^{h_2'}\bigr) = 0$, because when $S_2$ is not of minimal genus, the closed components removed during the annulus surgery are not necessarily tori.

    \item Case 2: if both $\partial((\widetilde{\mathcal{S}_1})_{0}^{-h-1})$ and $\partial((\widetilde{\mathcal{S}_2})_{0}^{h})$ are non-empty, i.e., they consist entirely of $I$-shaped annuli. We show that this case is impossible.

Since $(\widetilde{\mathcal{S}_1})_{0}^{-h-1}$ consists of $I$-shaped annuli, and by \cref{partial}, we have $\partial(\widetilde{\mathcal{S}_1})_{0}^{-h-1} = \widetilde{\mathcal{C}}_{0,-h-1} \cup \widetilde{\mathcal{C}}_{0,-h}$. These $I$-shaped annuli naturally establish a one-to-one correspondence between the components of $\widetilde{\mathcal{C}}_{0,-h-1}$ and $\widetilde{\mathcal{C}}_{0,-h}$.

Since $\widetilde{\mathcal{C}}_{0,-h-1}$ and $\widetilde{\mathcal{C}}_{0,-h}$ each consist of $T$ disjoint copies of $\mathbb{S}^1$, we write
\begin{align}
\widetilde{\mathcal{C}}_{0,-h-1}
&=: \bigcup_{t=1}^{T} (\widetilde{\mathcal{C}}_{0,-h-1})_t,
\qquad
\widetilde{\mathcal{C}}_{0,-h}
=: \bigcup_{t=1}^{T} (\widetilde{\mathcal{C}}_{0,-h})_t. \notag
\end{align}
Here each $(\widetilde{\mathcal{C}}_{0,-h-1})_t$ and $(\widetilde{\mathcal{C}}_{0,-h})_t$ is a simple loop. The same applies to $(\widetilde{\mathcal{S}_2})_{0}^{h}$, which also consists entirely of $I$-shaped annuli. Accordingly, we write $\widetilde{\mathcal{C}}_{0,h} =: \bigcup_{t=1}^{T} (\widetilde{\mathcal{C}}_{0,h})_t$, and $\widetilde{\mathcal{C}}_{0,h+1} =: \bigcup_{t=1}^{T} (\widetilde{\mathcal{C}}_{0,h+1})_t$.

By appropriately ordering these components, we can ensure that for any $1\leq t\leq T$, we have $(\widetilde{\mathcal{C}}_{0,-h-1})_t$ and $(\widetilde{\mathcal{C}}_{0,-h})_t$ are $\mathbb{S}^1$ components connected by an annulus $\widetilde{A}_t((\widetilde{\mathcal{S}_1})_{0}^{-h-1})\subset(\widetilde{\mathcal{S}_1})_{0}^{-h-1}$ (see \cref{fig:Ishapedannuli}). Moreover, for $h'\in\{h,h+1\}$, we have $p((\widetilde{\mathcal{C}}_{0,-h'})_t)=p((\widetilde{\mathcal{C}}_{h',0})_t)$. In this setting, there exists a bijection $\sigma:\{1,\cdots,T\} \rightarrow \{1,\cdots,T\}$ such that for each $1\leq t\leq T$, we have $(\widetilde{\mathcal{C}}_{h,0})_t$ and $(\widetilde{\mathcal{C}}_{h+1,0})_{\sigma(t)}$ are connected by an annulus $\widetilde{A}_t((\widetilde{\mathcal{S}_2})_{0}^{h}) \subset(\widetilde{\mathcal{S}_2})_{0}^{h}$.

From the properties of permutations, there exists a smallest positive integer $r$ such that $\sigma^r(1)=1$. Now we construct a torus $\mathbb{T}$:
\begin{equation}\label{T}
        \mathbb{T}:=p\left(\bigcup_{t=1}^{r}\left( \widetilde{A}_{\sigma^t(1)}((\widetilde{\mathcal{S}_1})_{0}^{-h-1}) \cup \widetilde{A}_{\sigma^t(1)}((\widetilde{\mathcal{S}_2})_{0}^{h}) \right)\right). 
\end{equation}

\begin{figure}[H]
\centering
\includegraphics[width=0.9\textwidth]{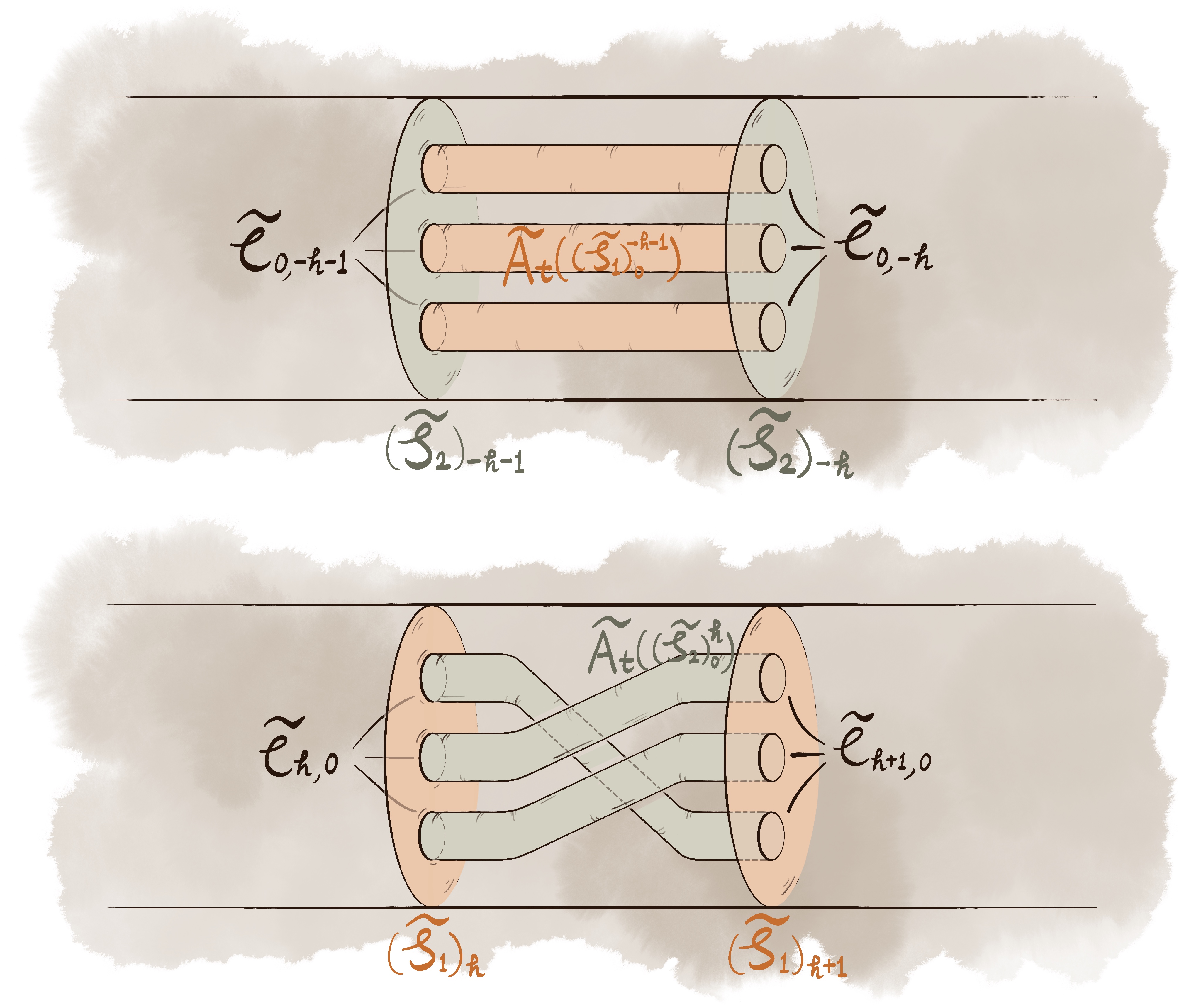}
\caption{Both $(\widetilde{\mathcal{S}_1})_{0}^{-h-1}$ and $(\widetilde{\mathcal{S}_2})_{0}^{h}$ consist entirely of $I$-shaped annuli.}
\label{fig:Ishapedannuli}
\end{figure}

We now prove the following properties of $\mathbb{T}$:
\begin{itemize}
    \item \textbf{Embeddedness}: $\mathbb{T}$ is an embedded torus.

    The collections $\left\{\widetilde{A}_{\sigma^t(1)}((\widetilde{\mathcal{S}_1})_{0}^{-h-1}) \right\}_{t=1}^{r}$ in $\widetilde{E}_{-h-1}(\mathcal{S}_2)$ and $\left\{\widetilde{A}_{\sigma^t(1)}((\widetilde{\mathcal{S}_2})_{0}^{h}) \right\}_{t=1}^{r}$ in $\widetilde{E}_{h}(\mathcal{S}_1)$ consist of pairwise disjoint annuli within their respective fundamental regions. Under the covering projection, these map to families of annuli in $\mathbb{S}^3\setminus W(K)$ that remain disjoint except for controlled intersections along $\mathcal{S}_1\cap \mathcal{S}_2$. The cyclic ordering induced by $\sigma$ ensures these projected annuli glue smoothly along the curves $p(\widetilde{\mathcal{C}}_{0,-h-1})$ and $p(\widetilde{\mathcal{C}}_{0,-h})$, forming the embedded torus $\mathbb{T}$. Crucially, the only intersections occur transversally at the surface intersections, and the disjointness in the cover prevents self-intersections, guaranteeing $\mathbb{T}$ is embedded.
    
    \item \textbf{Incompressibility}: $\mathbb{T}$ is incompressible.
    We only need to prove that the embedding map $\iota: \mathbb{T} \rightarrow \mathbb{S}^3 \setminus W(K) =: E$ induces an injective homomorphism $\iota_*: \pi_1(\mathbb{T}) \rightarrow \pi_1(E)$ on fundamental groups.

    For a point $x \in p((\widetilde{\mathcal{C}}_{0,-h-1})_1)$, the fundamental group $\pi_1(\mathbb{T}, x) \cong \mathbb{Z}^2$ has two generators $\alpha$ and $\beta$. Here, $\alpha$ represents $p((\widetilde{\mathcal{C}}_{0,-h-1})_1)$, while $\beta$ corresponds to the loop based at $x$ that traverses the following sequence of annuli and returns to $x$:
    \begin{equation}
        p\left( \widetilde{A}_{\sigma^{0}(1)}((\widetilde{\mathcal{S}_1})_{0}^{-h-1}) \right) , p\left( \widetilde{A}_{\sigma^{0}(1)}((\widetilde{\mathcal{S}_2})_{0}^{h}) \right), \cdots, p\left( \widetilde{A}_{\sigma^{r-1}(1)}((\widetilde{\mathcal{S}_1})_{0}^{-h-1}) \right) , p\left( \widetilde{A}_{\sigma^{r-1}(1)}((\widetilde{\mathcal{S}_2})_{0}^{h}) \right). \notag
    \end{equation}
    Recall that the abelianisation map is denoted by $\mathrm{ab}: \pi_1(E) \rightarrow H_1(E)$. The condition $\mathrm{ab} \circ \iota_*(\alpha)=0$ holds because $\alpha$ lifts to a closed loop in $\widetilde{E}$. In contrast, $\mathrm{ab} \circ \iota_*(\beta)\neq 0$ since $\beta$ admits a non-closed lift.
    
    Suppose that $\iota_*$ is not injective. Then there exists a nontrivial element $\gamma = a\alpha + b\beta \neq 0$ such that $\iota_*(\gamma)=0$. Applying the abelianisation map, we obtain
    \begin{align}
    \mathrm{ab} \circ \iota_*(\gamma) = a \cdot (\mathrm{ab} \circ \iota_*(\alpha)) + b \cdot (\mathrm{ab} \circ \iota_*(\beta)). \notag
    \end{align}
    Since $\mathrm{ab} \circ \iota_*(\gamma)=0$, $\mathrm{ab} \circ \iota_*(\alpha)=0$ and $\mathrm{ab} \circ \iota_*(\beta)\neq 0$, it follows that $b=0$.
    Hence $\gamma = a\alpha$, which implies that $\alpha$ is null-homotopic in $E$.
    
    This contradicts the incompressibility of the minimal-genus Seifert surface $\mathcal{S}_1$.
    Indeed, since $S_1$ has minimal genus, the surface $\mathcal{S}_1$ obtained from it also has minimal genus (see \cref{pro:minimal-genus}).
    
    \item \textbf{Non-boundary-parallelism}:

    The condition $\mathrm{ab} \circ \iota_*(\alpha) = 0$ forces $\alpha$ to be isotopic to a longitude on $\mathcal{S}_1$ or $\mathcal{S}_2$. For the case $r=1$, the torus $\mathbb{T}$ splits into two annuli $p\left(\widetilde{A}_{1}((\widetilde{\mathcal{S}_1})_{0}^{-h-1})\right) \subset \mathcal{S}_1$ and $p\left(\widetilde{A}_{1}((\widetilde{\mathcal{S}_2})_{0}^{h})\right) \subset \mathcal{S}_2$ by \cref{T}. Thus by \cref{lem: not boundary-parallel}, we conclude that $\mathbb{T}$ is not boundary-parallel.
    
    For the case $r > 1$, we observe that $\mathrm{ab} \circ \iota_*(\pi_1(\mathbb{T})) = r\mathbb{Z}$. However, if $\mathbb{T}$ were boundary-parallel, this would require $\mathrm{ab} \circ \iota_*(\pi_1(\mathbb{T})) = \mathbb{Z}$, leading to a contradiction.
\end{itemize}

The above property implies that $\mathbb{T}$ is an incompressible torus and not boundary-parallel, which contradicts the assumption in \cref{thm:euler-inequality} that $K$ is atoroidal. Therefore, such a case is in fact impossible.
\end{itemize}

\end{proof}

\section{Proof of \cref{main thm: Hyperbolic Knot Linear Bound}}

In this section, we complete the proof of \cref{main thm: Hyperbolic Knot Linear Bound}.

\begin{proposition}\label{pro: middle length}
Let $K \subset \mathbb{S}^3$ be an atoroidal knot, and let $(S_1,S_2)$ be a favourable pair of Seifert surfaces. Denote by $g = g(K)$ the genus of $S_1$, and by $\ell$ the genus of $S_2$. Let $\{h_3,\ldots,h_4\}$ be the central zone of $(S_1,S_2)$. Then
\begin{align}
    h_4 - h_3 \leq g + \ell -1.
\end{align}
\end{proposition}

\begin{proof}
By \cref{thm:euler-inequality}, for any integer $h_3 \leq h \leq h_4$, we have
\begin{align}
\widehat{\chi}(h) := \chi((\widetilde{S_1})_{0}^{-h-1}) + \chi((\widetilde{S_2})_{0}^{h}) \leq -1, \notag
\end{align}
or $\left((\widetilde{S_1})_{0}^{-h-1} \cup (\widetilde{S_2})_{0}^{h} \right) \cap \partial \widetilde{E} \neq \emptyset$.

We aim to show that for any $h_3 \leq h \leq h_4$, either $\widehat{\chi}(h) \leq -2$, or at least one of $(\widetilde{S_1})_{0}^{-h-1}$ and $(\widetilde{S_2})_{0}^{h}$ has a boundary component lying in $\partial \widetilde{E}$.

If $\chi\bigl((\widetilde{S_2})_{0}^{h}\bigr)=0$ and both $(\widetilde{S_1})_{0}^{-h-1}$ and $(\widetilde{S_2})_{0}^{h}$ have no boundary components lying in $\partial \widetilde{E}$, then $(\widetilde{S_2})_{0}^{h}$ is a union of annuli, which may include both $I$-shaped and $U$-shaped annuli. It follows that the difference between the numbers of boundary components of $(\widetilde{S_2})_{0}^{h}$ lying on $\widetilde{C}_{h,0}$ and on $\widetilde{C}_{h+1,0}$ is an even integer, since each $U$-shaped annulus contributes two boundary components on the same side. We now prove that $\chi((\widetilde{S_1})_{0}^{-h-1}) \leq -2$. Otherwise, if $\chi((\widetilde{S_1})_{0}^{-h-1}) = -1$, then $(\widetilde{S_1})_{0}^{-h-1}$ must contain a unique pair-of-pants component. Since $(\widetilde{S_1})_{0}^{-h-1}$ has no boundary components lying in $\partial \widetilde{E}$, this implies that the difference in the number of branches between $\widetilde{C}_{0,-h-1}$ and $\widetilde{C}_{0,-h}$ is odd, contradicting the fact that the difference between the number of branches in $\widetilde{C}_{h,0}$ and $\widetilde{C}_{h+1,0}$ is even.

By symmetry, if $\chi\bigl((\widetilde{S_1})_{0}^{-h-1}\bigr)=0$ and both $(\widetilde{S_1})_{0}^{-h-1}$ and $(\widetilde{S_2})_{0}^{h}$ have no boundary components lying in $\partial \widetilde{E}$, then $\chi((\widetilde{S_2})_{0}^{h}) \leq -2$. Therefore, either $\widehat{\chi}(h) \leq -2$, or at least one of $(\widetilde{S_1})_{0}^{-h-1}$ and $(\widetilde{S_2})_{0}^{h}$ has a boundary component lying in $\partial \widetilde{E}$, as claimed.

If $(\widetilde{S_2})_{0}^{h}$ has a boundary component lying in $\partial \widetilde{E}$, then, since $(S_1,S_2)$ is a favourable pair (and hence every intersection loop on $S_2$ other than $K$ is not boundary-parallel; see \cref{defi: Favourable pair}), the connected component of $(\widetilde{S_2})_{0}^{h}$ containing this boundary component cannot be an annulus. Consequently, $\chi\bigl((\widetilde{S_2})_{0}^{h}\bigr)\leq -1$, and hence $\widehat{\chi}(h)\leq -1$.

It follows that for $h \in [h_3,h_4]$, there are at most two exceptional values of $h$ for which $\widehat{\chi}(h) > -2$. More precisely, there exists one value of $h$ such that $(\widetilde{S_2})_{0}^{h}$ has a boundary component lying in $\partial \widetilde{E}$, in which case $\widehat{\chi}(h) \leq -1$, and possibly a different value of $h$ such that $(\widetilde{S_1})_{0}^{-h-1}$ has a boundary component lying in $\partial \widetilde{E}$, in which case $\widehat{\chi}(h) \leq 0$. For all remaining values of $h$, we have $\widehat{\chi}(h) \leq -2$. Therefore,
\begin{align} \label{eq1.1}
    \sum_{h=h_3}^{h_4} \widehat{\chi}(h) \leq -2 ((h_4-h_3+1)-2) -1 = -2 (h_4-h_3 - \frac{1}{2}).
\end{align}

Since $\chi((\widetilde{S_1})_{0}) = 2-2g-1=1-2g$ and $\chi((\widetilde{S_2})_{0}) = 2 - 2\ell - 1 = 1-2\ell$, it follows that
\begin{align}\label{eq1.2}
\sum_{h=h_3}^{h_4} \widehat{\chi}(h) \geq (1-2g)+(1-2\ell) = 2(1-g-\ell).
\end{align}

Combining \cref{eq1.1} and \cref{eq1.2}, we obtain
\begin{align}
h_4 - h_3 \leq \frac{1}{2} - \frac{1}{2} \sum_{h=h_3}^{h_4} \widehat{\chi}(h) 
\leq \frac{1}{2} - \frac{1}{2}(2(1-g-\ell)) = g+\ell - \frac{1}{2}. \notag
\end{align}
Since $h_3,h_4,g,\ell$ are integers, we have $h_4 - h_3 \leq  g+\ell - 1$, which completes the proof of the lemma.
\end{proof}

\begin{theorem}\label{mainthm}
Let $K \subset \mathbb{S}^3$ be an atoroidal knot, and let $(S_1,S_2)$ be a favourable pair. Denote by $g = g(K)$ the genus of $S_1$, and by $\ell$ the genus of $S_2$. Denote $\left\{ h \in \mathbb{Z} \;\middle|\; (\widetilde{S_2})_0^h \neq \emptyset \right\} =: \{ h_1, h_1+1, \cdots, h_2 \}$. Then
\begin{align}
    h_2 - h_1 \leq \max\{4\ell + g - 3, 3\ell + 3g - 5\}.
\end{align}
\end{theorem}

\begin{proof}
    By \cref{thm:euler-inequality} we have $\chi\bigl((\widetilde{S_1})_0^{-h-1}\bigr)=0$ for all $h \geq h_4+1$. If $h_4+1 < h_2$, since $(\widetilde{S_1})_0$ is connected, there exists an annulus $\widetilde{A}_1 \subset (\widetilde{S_1})_0$ whose two boundary components lie on $(\widetilde{S_2})_{-(h_4+1)}$ and $(\widetilde{S_2})_{-h_2}$, respectively (see the upper panel of \cref{fig:A1A2A3}). If $h_4 + 1 = h_2$, we define $\widetilde{A}_1$ to be an arbitrarily chosen connected component of $(\widetilde{S_1})_0 \cap (\widetilde{S_2})_{-h_2}$, viewed as a degenerate annulus consisting of a single loop. In either case, $\widetilde{A}_1$ intersects $(\widetilde{S_2})_{-h}$ in a collection of loops for each $h$ with $h_4+1 \leq h \leq h_2$. For each such $h$, we choose an arbitrary loop $\widetilde{C}_h \subset \widetilde{A}_1 \cap (\widetilde{S_2})_{-h}$, and define the collection of chosen loops
    \begin{align}
    \widetilde{\mathcal{E}}_1 := \Bigl\{ \widetilde{C}_h \ \Big|\ h_4+1 \leq h \leq h_2 \Bigr\}. \notag
    \end{align}
    Then
    \begin{align}
    \label{eq:C_1}
    \#\widetilde{\mathcal{E}}_1 = h_2 - h_4.
    \end{align}
    If $h_4 = h_2$, we set $\widetilde{A}_1 := \varnothing$ and $\widetilde{\mathcal{E}}_1 := \varnothing$.

    Similarly, if $h_1 + 1 < h_3$, there exists an annulus $\widetilde{A}_2 \subset (\widetilde{S_1})_0$ whose two boundary components lie on $(\widetilde{S_2})_{-(h_1+1)}$ and $(\widetilde{S_2})_{-h_3}$, respectively (see the upper panel of \cref{fig:A1A2A3}). If $h_1 + 1 = h_3$, we define $\widetilde{A}_2$ to be an arbitrarily chosen connected component of $(\widetilde{S_1})_0 \cap (\widetilde{S_2})_{-h_3}$. In either case, $\widetilde{A}_2$ intersects $(\widetilde{S_2})_{-h}$ for all $h$ with $h_1+1 \leq h \leq h_3$. As before, we choose an arbitrary loop $\widetilde{C}_h \subset \widetilde{A}_2 \cap (\widetilde{S_2})_{-h}$ for each $h$, and define
    \begin{align}
    \widetilde{\mathcal{E}}_2 := \Bigl\{ \widetilde{C}_h \ \Big|\ h_1+1 \leq h \leq h_3 \Bigr\}. \notag
    \end{align}
    Then
    \begin{align}
    \label{eq:C_2}
    \#\widetilde{\mathcal{E}}_2 = h_3 - h_1.
    \end{align}
    If $h_1 = h_3$, we set $\widetilde{A}_2 := \emptyset$ and $\widetilde{\mathcal{E}}_2 := \varnothing$.

    We define
    \begin{align}
    \widetilde{\mathcal{E}} := \widetilde{\mathcal{E}}_1 \cup \widetilde{\mathcal{E}}_2. \notag
    \end{align}

    \textbf{Claim~1.} There do not exist two distinct elements in $\widetilde{\mathcal{E}}_1$ whose images under the covering projection are isotopic on $S_2$. Similarly, there do not exist two distinct elements in $\widetilde{\mathcal{E}}_2$ whose images under the covering projection are isotopic on $S_2$.
    
    We argue by contradiction. For the first statement, suppose that there exist two distinct elements $\widetilde{C}_i, \widetilde{C}_j \in \widetilde{\mathcal{E}}_1$ such that their images $p(\widetilde{C}_i)$ and $p(\widetilde{C}_j)$ are isotopic on $S_2$. Then $p(\widetilde{C}_i)$ and $p(\widetilde{C}_j)$ bound annuli on both $S_1$ and $S_2$, and these two annuli together form an embedded torus. Using the same argument as in Case~2 of the proof of \cref{thm:euler-inequality},
    together with \cref{lem: not boundary-parallel}, we conclude that this torus is essential. This contradicts the assumption that $K$ is atoroidal. Hence the assumption is false, and claim~1 follows. The argument for $\widetilde{\mathcal{E}}_2$ is identical.

    Moreover, since $(S_1,S_2)$ is a favourable pair, every element of $\widetilde{\mathcal{E}}_1$ and $\widetilde{\mathcal{E}}_2$ projects to an essential loop on $S_2$. Since there are at most $3\ell-2$ pairwise non-isotopic, disjoint essential loops on $S_2$ (as follows from a pants decomposition), we obtain $\#\widetilde{\mathcal{E}}_1 \leq 3\ell-2$, and $\#\widetilde{\mathcal{E}}_2 \leq 3\ell-2$. Combining this with \cref{eq:C_1} and \cref{eq:C_2}, we have
    \begin{align}
        \label{main eq}
        h_2 - h_4 \leq 3\ell - 2, \quad h_3 - h_1 \leq 3\ell - 2.
    \end{align}

    In fact, at this stage we have already shown that $h_2 - h_1$ is bounded by a linear function of $\ell$. More precisely, combining \cref{main eq} and \cref{pro: middle length}, we obtain
    \begin{align}
        h_2 - h_1 &= (h_4 - h_3) + (h_2-h_4) + (h_3-h_1) \notag \\
        &\leq (g + \ell - 1) + (3\ell - 2) + (3\ell - 2) \notag \\
        &\leq 7\ell + g - 5. \notag
    \end{align}
    However, in order to obtain a better constant, we need to carry out a more refined analysis, as follows.

    We first consider a special case. Suppose that $\widetilde{A}_1 = \varnothing$ or $\widetilde{A}_2 = \varnothing$ (or both), so that $h_2 - h_4 = 0$ whenever $\widetilde{A}_1 = \varnothing$, and $h_3 - h_1 = 0$ whenever $\widetilde{A}_2 = \varnothing$. Then we have 
    \begin{align}
        h_2 - h_1 \leq (g+\ell-1) + (3\ell-2) + 0 \leq 4\ell + g - 3, \notag
    \end{align}
    and hence the theorem follows in this case. Therefore, in the remainder of the proof we assume that $\widetilde{A}_1 \neq \varnothing$ and $\widetilde{A}_2 \neq \varnothing$. Moreover,
    \begin{align}
        h_1 + 1 \leq h_3 \quad \text{and} \quad h_4 + 1 \leq h_2. \notag
    \end{align}

    WLOG, we may assume that $h_2 - h_1$ is sufficiently large, namely $h_2 - h_1 \geq 3g - 1$. Otherwise, we have $h_2 - h_1 \leq (3g - 1) - 1 \leq 3g + 3\ell - 5$, and hence the theorem follows immediately. Thus, we can choose integers $h_5$ and $h_6$ such that
    \begin{align}
    h_1+1 \leq h_5 \leq h_3, \quad h_4+1 \leq h_6 \leq h_2, \notag
    \end{align}
    and
    \begin{align}
        \label{main eq 3}
    h_6-h_5 = \max\{h_4+1-h_3,3g-2\}.
    \end{align}
    We define
    \begin{align}
    \widetilde{\mathcal{E}}' := \bigl\{ \widetilde{C}_h \in \widetilde{\mathcal{E}} \ \big|\ h \leq h_5 \ \text{or} \ h \geq h_6 \bigr\}. \notag
    \end{align}

    \begin{figure}[H]
    \centering
    \includegraphics[width=0.9\textwidth]{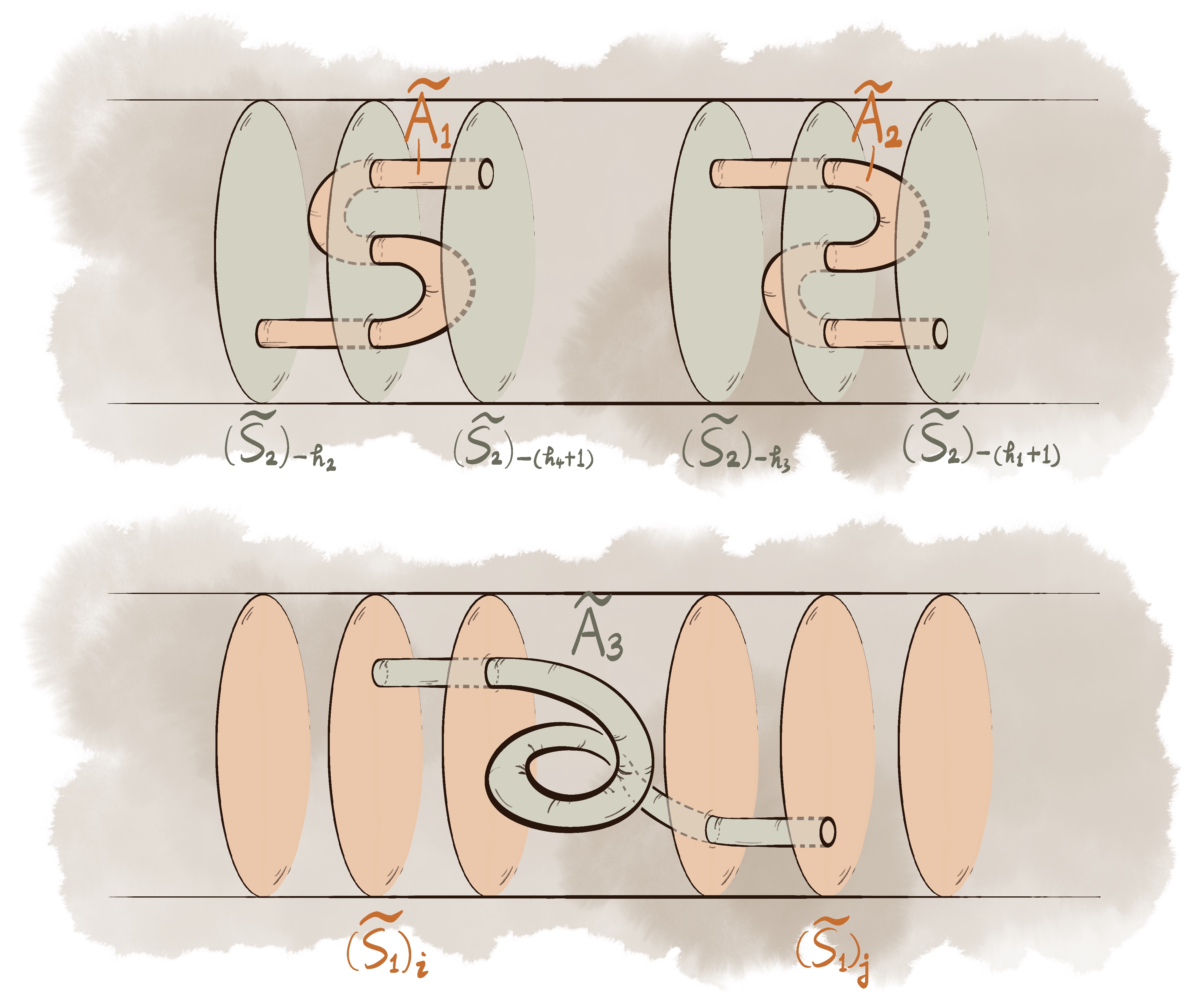}
    \caption{The annuli $\widetilde{A}_1$, $\widetilde{A}_2$, and $\widetilde{A}_3$.}
    \label{fig:A1A2A3}
    \end{figure}

    \textbf{Claim~2.} There do not exist two distinct elements in $\widetilde{\mathcal{E}}'$ whose images under the covering projection are isotopic on $S_2$.
    
    We again argue by contradiction. Suppose that there exist two distinct elements $\widetilde{C}_i, \widetilde{C}_j \in \widetilde{\mathcal{E}'}$ such that they are isotopic on $S_2$. If both $\widetilde{C}_i$ and $\widetilde{C}_j$ lie in $\widetilde{\mathcal{E}}_1$, then we obtain a contradiction by Claim~1. The same argument applies if both lie in $\widetilde{\mathcal{E}}_2$. Hence, it suffices to consider the situation where $\widetilde{C}_i$ and $\widetilde{C}_j$ lie in different annuli. Without loss of generality, assume that $\widetilde{C}_i \in \widetilde{\mathcal{E}}_1$ and $\widetilde{C}_j \in \widetilde{\mathcal{E}}_2$.

    Since
    \begin{align}
        \widetilde{C}_i = \widetilde{A}_1 \cap (\widetilde{S}_2)_{-i} \subset (\widetilde{S}_1)_0 \cap (\widetilde{S}_2)_{-i}, \quad \widetilde{C}_j = \widetilde{A}_2 \cap (\widetilde{S}_2)_{-j} \subset (\widetilde{S}_1)_0 \cap (\widetilde{S}_2)_{-j}, \notag
    \end{align}
    it follows that
    \begin{align}
        \tau^{i}(\widetilde{C}_i) \subset (\widetilde{S}_1)_{i} \cap (\widetilde{S}_2)_0, \quad \tau^{j}(\widetilde{C}_j) \subset (\widetilde{S}_1)_{j} \cap (\widetilde{S}_2)_0. \notag
    \end{align}
    Since $p(\widetilde{C}_i)$ and $p(\widetilde{C}_j)$ are isotopic on $S_2$, the loops $\tau^{i}(\widetilde{C}_i)$ and $\tau^{j}(\widetilde{C}_j)$ bound an annulus in $(\widetilde{S}_2)_0$, which we denote by $\widetilde{A}_3$ (see the bottom panel of \cref{fig:A1A2A3}). Moreover, $\widetilde{A}_3$ intersects $(\widetilde{S}_1)_h$ for all $i \leq h \leq j$, and these intersections consist of $j-i+1$ loop components.

    Since $i \leq h_5 \leq h_6 \leq j$, we have $j - i + 1 \geq h_6 - h_5 + 1 \geq 3g - 1 > 3g - 2$. That is, $j - i + 1$ exceeds the maximal possible number of pairwise non-isotopic, disjoint essential loops on $S_1$. Therefore, there must exist $i \leq t_1 \neq t_2 \leq j$ such that $p(\widetilde{A}_3 \cap (\widetilde{S}_1)_{t_1})$ and $p(\widetilde{A}_3 \cap (\widetilde{S}_1)_{t_2})$ are isotopic on $S_1$. Consequently, we obtain two annuli that together form an essential torus, contradicting the assumption that $K$ is atoroidal. This proves Claim~2.

    As a result, we obtain $\# \widetilde{\mathcal{E}'} = \bigl(h_5 - (h_1 + 1) + 1\bigr) + \bigl(h_2 - h_6 + 1\bigr) \leq 3\ell - 2$, i.e.
    \begin{align}
        \label{main eq 4}
        (h_2-h_1) - (h_6-h_5) \leq 3\ell - 3.
    \end{align}

    Combining \cref{main eq 3} and \cref{main eq 4}, we obtain
    \begin{align}
        h_2 - h_1 &\leq (h_6 - h_5) + 3\ell - 3 \notag\\
        &\leq \max\{h_4+1-h_3,3g-2\} + 3\ell - 3 \notag\\
        &\leq \max\{g+\ell,3g-2\} + 3\ell - 3 \notag\\
        &=\max\{4\ell + g - 3, 3\ell + 3g - 5\}. \notag
    \end{align}
    This completes the proof of the theorem.

\end{proof}

\begin{corollary}\label{mainresult1}
    Let $K \subset \mathbb{S}^3$ be an atoroidal knot, and let $S_1$ be a minimal-genus Seifert surface and $S_2$ be an incompressible Seifert surface. Denote by $g = g(K)$ the genus of $S_1$, and by $\ell$ the genus of $S_2$. Then
\begin{align}
    d_{\ell}([S_1],[S_2]) = d_{\infty}([S_1],[S_2]) \leq \max\{4\ell + g - 2, 3\ell + 3g - 4\}. \notag
\end{align}
\end{corollary}

\begin{proof}
    By \cref{Sakuma3}, up to isotopy we may assume that $(S_1,S_2)$ is a favourable pair.
    
    By \cref{defi: Kakimizu distance} and \cref{kakimizu1}, together with \cref{mainthm}, we obtain
    \begin{align}
        d_{\ell}([S_1],[S_2]) &= d_{\infty}([S_1],[S_2]) = d_*([S_1],[S_2]) \notag \\
        &\leq d_*(S_1,S_2)
         = \# \left\{ h \in \mathbb{Z} \;\middle|\; (\widetilde{S_2})_0^h \neq \emptyset \right\}
         = h_2 - h_1 + 1 \notag \\
         &\leq \max\{4\ell + g - 2, 3\ell + 3g - 4\}. \notag
    \end{align}
\end{proof}

Taking $\ell = g$, we obtain
\begin{align}
d_g([S_1],[S_2]) = d_\infty([S_1],[S_2]) \leq
\begin{cases}
6g - 4, & g \geq 2,\\
3, & g = 1.
\end{cases} \notag
\end{align}
On the other hand, Sakuma--Shackleton showed that when $g = 1$, the diameter of the Kakimizu complex satisfies $\operatorname{diam}(MS(K)) \leq 2$ \cite[Corollary~1.3]{MR2531146}.
Combining these results, we obtain the following corollary.

\begin{corollary}[\cref{main thm: Hyperbolic Knot Linear Bound}]\label{mainresult2}
    Let $K \subset \mathbb{S}^3$ be an atoroidal knot, and let $S_1$ and $S_2$ be two minimal-genus Seifert surfaces for $K$. Denote by $g=g(K)$ the genus of $S_1$ and $S_2$. Then
\begin{align}
    d_{g}([S_1],[S_2]) = d_{\infty}([S_1],[S_2]) \leq 6g-4. \notag
\end{align}
Hence, by the arbitrariness of $S_1$ and $S_2$, we have
\begin{align}
    \operatorname{diam} (MS(K)) \leq 6g-4.
\end{align}
\end{corollary}

\begin{corollary}[\cref{main thm2: Hyperbolic Knot Linear Bound}]\label{mainresult3}
    Let $K \subset \mathbb{S}^3$ be an atoroidal knot, and let $[S_1]$ and $[S_2]$ be two vertices in $IS_{\ell}(K)$. Let $g$ be the genus of the knot $K$. Then
\begin{align}
    d_{\ell}([S_1],[S_2]) \leq 2\max\{4\ell + g - 2, 3\ell + 3g - 4\}, \notag
\end{align}
and
\begin{align}
d_{\infty}([S_1],[S_2]) \leq 2\max\{4\ell + g - 2, 3\ell + 3g - 4\}. \notag
\end{align}
Hence,
\begin{align}
    \operatorname{diam} (IS_{\ell}(K)) \leq 2\max\{4\ell + g - 2, 3\ell + 3g - 4\}.
\end{align}
\end{corollary}

\begin{proof}
    Choose a minimal-genus Seifert surface $S_3$ for $K$. We know that $[S_3]$ is also a vertex of $IS_{\ell}(K)$. By \cref{mainresult1}, we have
    \begin{align}
        d_{\ell}([S_1],[S_2]) \leq d_{\ell}([S_1],[S_3]) + d_{\ell}([S_3],[S_2]) \leq 2\max\{4\ell + g - 2, 3\ell + 3g - 4\}. \notag
    \end{align}
    The same argument applies to $d_{\infty}$. This completes the proof.
\end{proof}

\section*{Acknowledgements}
We express deep gratitude to Professors Shing-Tung Yau and Yi Huang for their invaluable guidance and consistent assistance throughout this work, and to Makoto Sakuma for his insightful suggestions and thought-provoking questions that significantly shaped the research.

\printbibliography[title={References}] 


\end{document}